\documentclass[12pt]{amsart}

\usepackage{amsmath,amstext,amssymb,amsopn,amsthm}
\usepackage{url,verbatim}
\usepackage{mathtools}
\usepackage{enumerate}

\usepackage{color,graphicx}

\usepackage[margin=30mm]{geometry}
\usepackage{eucal,mathrsfs,dsfont}

\usepackage{subfig}
\usepackage{graphicx}

\usepackage[section]{placeins}

\allowdisplaybreaks

\newtheorem{theorem}{Theorem}[section]

\theoremstyle{definition}
\newtheorem{conjecture}[theorem]{Conjecture}

\newtheorem{problem}[theorem]{Problem}
\newtheorem{remark}[theorem]{Remark}

\theoremstyle{remark}

\numberwithin{equation}{section}

\newcommand{\eps}{\varepsilon}

\newcommand{\calZ}{\mathcal{Z}}
\newcommand{\calY}{\mathcal{Y}}

\newcommand{\calX}{\mathcal{X}}

\newcommand{\E}{\operatorname{\mathds{E}}} 
\renewcommand{\P}{\operatorname{\mathds{P}}} 
\newcommand{\R}{\mathds{R}}

\newcommand{\prt}{\partial}

\newcommand{\wh}{\widehat}
\newcommand{\wt}{\widetilde}

\DeclareMathOperator{\Var}{Var}

\title[Ballistic aggregation]{Ballistic aggregation displays \\ self-organized criticality}
\author{Krzysztof Burdzy}

\address{Department of Mathematics, Box 354350, University of Washington, Seattle, WA 98195}
\email{burdzy@uw.edu}

\thanks{Research supported in part by Simons Foundation Grant 928958. The author would like to thank the Isaac Newton Institute for Mathematical Sciences, Cambridge, for support and hospitality during the program ``Stochastic systems for anomalous diffusion'' where work on this paper was undertaken. This work was supported by EPSRC grant no EP/R014604/1.}

\begin{document}

\begin{abstract}

Consider the convex hull of a collection of disjoint open discs with radii $1/2$. 
The boundary of the convex hull consists of a finite number of line segments and arcs. Randomly choose a point in one of the arcs in the boundary so that the density of its distribution is proportional to the arc measure. 
Attach a new disc at the chosen point so that it is outside of the convex hull and tangential to its boundary. Replace the original convex hull with the convex hull of all preexisting discs and the new disc. Continue in the same manner.
Simulations show that disc clusters form long, straight, or slightly curved filaments with small side branches and occasional macroscopic side branches. The shape of the convex hull is either an equilateral triangle or a quadrangle. Side branches play the role analogous to avalanches in sandpile models, one of the best-known examples of self-organized criticality (SOC). Our results indicate that the size of a branch obeys a power law, as expected of avalanches in sandpile models and similar ``catastrophies'' in other SOC models.

\end{abstract}

\maketitle

\section{Introduction}\label{intro}

We will discuss an accumulation model that sits at the end of the spectrum of models introduced in \cite{stefan}.
Its inspiration is the famous ``diffusion-limited aggregation'' (DLA) model, but its mechanism is, in a sense, the opposite of the DLA accumulation scheme.

The initial cluster consists of a finite number of disjoint open discs in $\R^2$ with radii $1/2$.
Let $C_0$ be the convex hull of these discs. The boundary of $C_0$ consists of a finite number of line segments and arcs. We randomly choose a point in $\prt C_0$ on one of the arcs so that the density of its distribution is proportional to the arc measure. 
We attach a new disc at the chosen point so that it is outside $C_0$ and tangential to $\prt C_0$. We let $C_1$ be the convex hull of $C_0$ and the new disc. We continue in the same manner and construct $C_n$ for $n\geq 0$. See Fig. \ref{fig_small} for a sample of small cluster simulations.

\begin{figure}
    \centering
    \begin{minipage}{0.33\textwidth}
        \centering
        \includegraphics[width=0.99\textwidth]{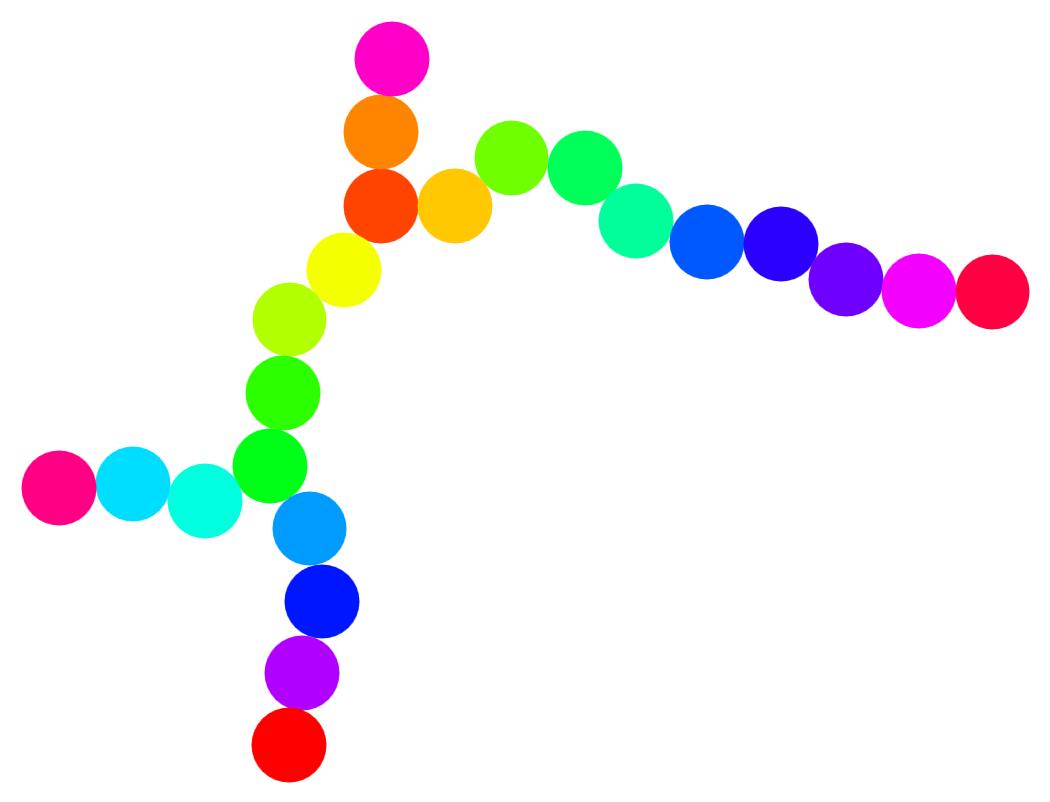}
    \end{minipage}\hfill
    \begin{minipage}{0.33\textwidth}
        \centering
        \includegraphics[width=0.99\textwidth]{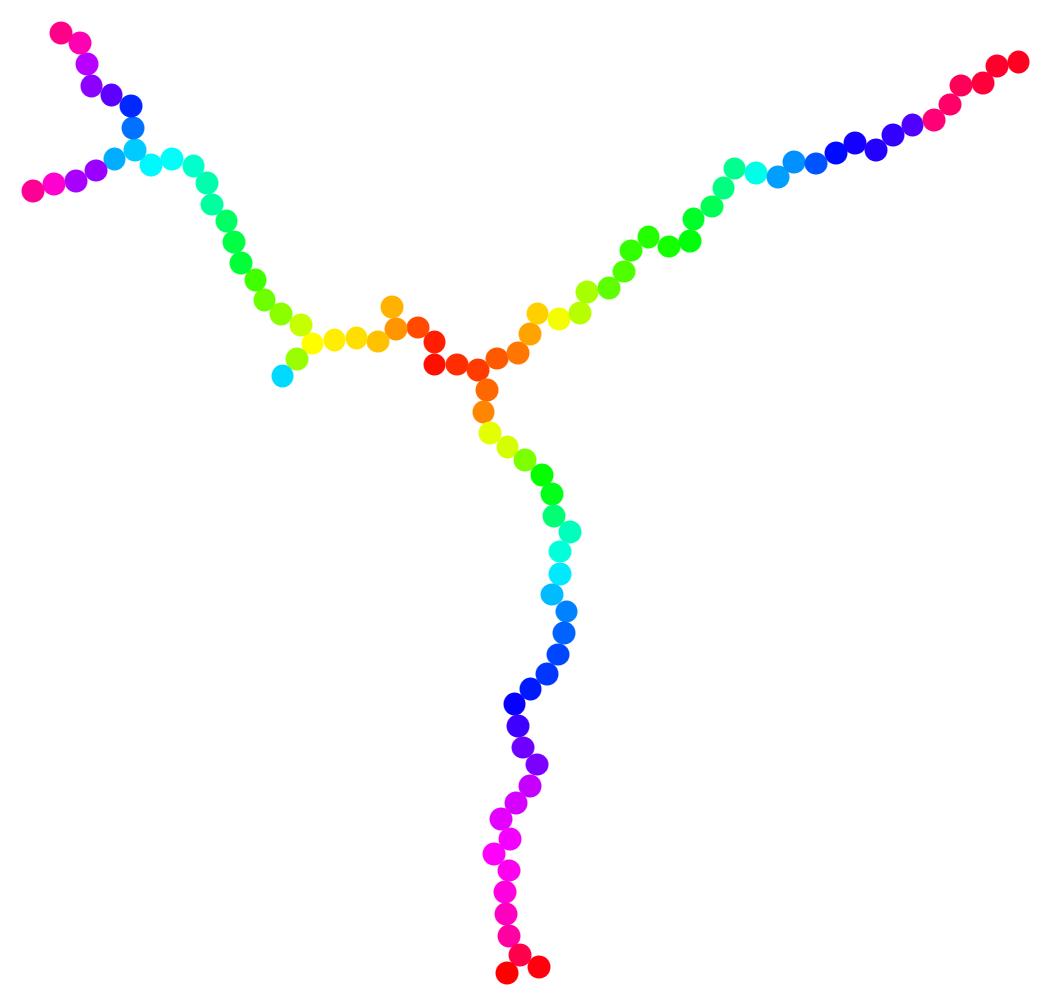} 
    \end{minipage}
    \begin{minipage}{0.33\textwidth}
        \centering
        \includegraphics[width=0.99\textwidth]{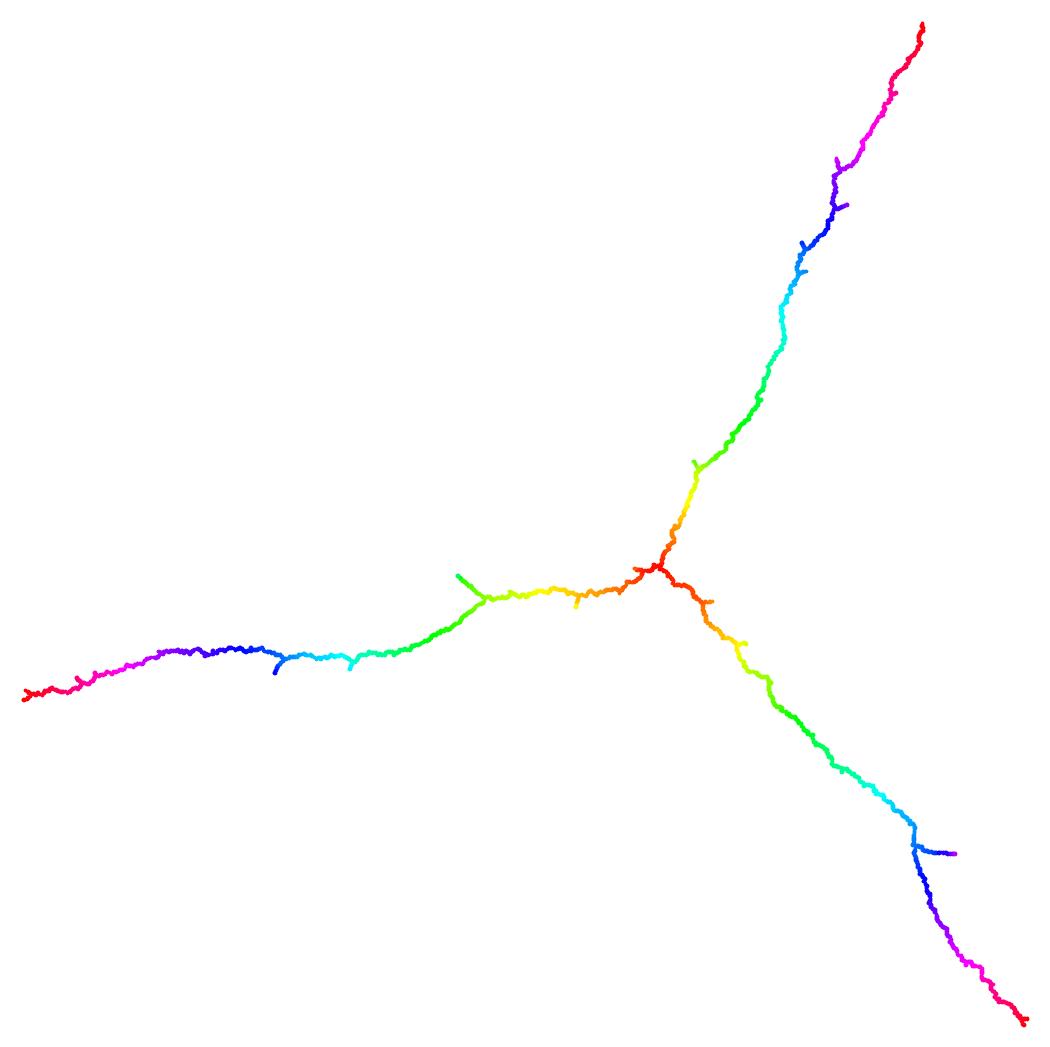}  
    \end{minipage}
    \caption{Simulations of clusters. Initial clusters consisted of single discs. Colors indicate the order of attachment. The numbers of discs in the clusters are 20, 100 and 1,000. } \label{fig_small}
\end{figure}

Heuristically, the construction is equivalent to placing a disc at a very large distance from the origin (``at infinity'') at a uniform angle and moving the disc along a straight line towards the closest point in $\prt C_n$. When the disc hits $C_n$, it stops and is attached to the cluster. The set $C_{n+1}$ is defined as the convex hull of $C_n$ and the new disc.

Computer simulations suggest that for large $n$, the clusters consist of three or four macroscopic (i.e., comparable in size with the whole cluster) arms. If there are three arms, they are almost straight, almost equally long, forming angles close to $2\pi/3$, and possessing many small side branches. Sometimes, some of the side branches are macroscopically large. Side branches play the role analogous to avalanches in sandpile models, one of the best-known examples of self-organized criticality (SOC). Our results indicate that the size of a branch obeys a power law, as expected of avalanches in sandpile models and similar ``catastrophies'' in other SOC models.

\subsection{History of growth models}
The best-known random growth model is Diffusion Limited Aggregation (DLA), introduced by Witten \& Sanders \cite{witten}. DLA models random growth on the standard two-dimensional lattice by assuming that a particle coming from `infinity' performs a random walk until it first touches an already existing set of particles, at which point it gets stuck for all time. The main results are due to Kesten \cite{ kesten1, kesten2} who proved that the diameter of a cluster of $n$ particles under DLA on the lattice $\mathbb{Z}^2$ is $\lesssim n^{2/3}$ (along with generalizations to $\mathbb{Z}^d$). No non-trivial lower bound (better than $\gtrsim n^{1/2}$) is known.
Only a handful of rigorous results about DLA are known. Among them is the proof of the existence of infinitely many holes in the two-dimensional DLA cluster given in \cite{eberz}.

Many other models have been proposed, including the Eden model \cite{eden},  the Vold--Sutherland model \cite{suh, vold}, the Dielectric Breakdown Model (DBM) \cite{nie} and the Hastings-Levitov model \cite{hastings}. There are only a few rigorous results. 

\subsection{History of self-organized criticality}
The idea of self-organized criticality was introduced by Per Bak, Chao Tang and Kurt Wiesenfeld in  \cite{BTW87,BTW88}. The books \cite{Bak,Jensen} remain the standard introductions. For a more recent monograph, see \cite{Prue}.

The most popular SOC models are sandpiles and forest fires. We list here only a few references for sandpiles \cite{sand1,sand2} and forest fires \cite{forest1,forest2,forest3}, to serve as starting points for the search of an extensive bibliography.

\subsection{Organization of the paper}
The next section presents simulation results and offers some general observations. Section \ref{o30.12} sets up a model for the growth of one branch in the cluster of discs. This is followed by Section \ref{o31.1}, which contains a calculation of the transition probabilities for the Markov chain defined in Section \ref{o30.12}. Section \ref{secOU} presents estimates for the distribution of the escape time from an interval for the continuous-time counterpart of the model 
introduced in Section \ref{o30.12}. Section \ref{n1.5} discusses the consequences of the escape time estimates. Section \ref{n1.6} contains a comparison of our ``time-dependent anti-Ornstein-Uhlenbeck process'' with the classical Ornstein-Uhlenbeck process. Section \ref{n1.7} contains a sketch of a proof that the asymptotic shape of the convex hull cannot be a polygon with distinct obtuse angles.

\section{Simulation results with discussion}\label{o30.2}

Fig. \ref{fig_basic} shows a sample of clusters with different sizes, from $10^4$ discs to $10^6$ discs. Our simulation results suggest that, when the number of discs grows, the shape of the convex hull becomes either close to the equilateral triangle or a quadrangle.  
Fig. \ref{fig_small} suggests that the triangular shape of the convex hull emerges already in clusters of size 100.

We will call a disc ``extremal'' if it touches the boundary of the convex hull $\prt C_n$.
The number of extremal discs is typically very small. In each of the three simulations of clusters with $10^6$ discs presented in the bottom row of Fig. \ref{fig_basic}, the number of discs on the boundary was 9. In a (non-pictured) sequence of five simulations of clusters with $10^5$ discs, the numbers of discs on the boundary were 7, 6, 6, 5, and 8. This strongly indicates that the diameter of $C_n$ grows at a linear speed in $n$.
Hence, our ballistic aggregation model lies at the other end of aggregation models than the Eden model (\cite{eden}) because it was proved in \cite{DHL} that the diameter of an Eden cluster has size $n^{1/2}$.

The stability of the shape of the convex hull can be explained as follows. Suppose there are three or four main straight components of $\prt C_n$. Consider two adjacent line segments $I_1$ and $I_2$, in the sense that two of their endpoints are much closer than the cluster's diameter. Let $A$ be the part of $\prt C_n$ between the closest endpoints of $I_1$ and $I_2$. The rate at which discs are attached to $A$ depends on the angle between $I_1$ and $I_2$. The smaller the angle, the higher the rate. At the same time, the cluster shows a branching tendency close to $A$ that also depends on the angle. When the angle is small, the branching rate is high. If the cluster has a fork close to $A$, the growth in this direction will slow down because the two (or more) arms of the fork have to grow at the same time. 
Hence, there are two contradictory mechanisms that could either slow down or speed up acute corners. Simulations show that the second mechanism appears to win, i.e.,
if $\prt C_n$ is approximately polygonal, the pointed corners tend to slow down. If there are three main straight components of $\prt C_n$, the angles between them revert to $\pi/3$. 

If there are four main straight components of $\prt C_n$, the angles are not stable. The largest angle tends to grow until it reaches $\pi/2$ and then the adjacent macroscopic line segments in $\prt C_n$ merge into one. 
The simulations of clusters with $10^6$ discs presented in the bottom row of  \ref{fig_basic} raise a question of whether macroscopic side branches will occur on a large scale. In other words, 

\begin{problem}\label{o30.1}
    For a large number of discs, how likely is it to see (at least) four rather than three macroscopic branches? 
\end{problem}

Regrettably, we cannot prove any results on our ballistic aggregation model at this time. 
We separated Problem \ref{o30.1} from other problems listed below because this is the only problem on which we can shed light via rigorous theorems. 
We list several other problems of various levels of generality and difficulty below.
Some of these problems are subproblems of other problems on the list. We state them separately because we believe that subproblems might be more tractable than the most general problems.

\begin{problem}\label{o30.2a}
    (i) Let $C_n'$ be the convex hull $C_n$ rescaled so that its diameter is 1. Find the limit of distributions of $\prt C_n'$ as $n\to\infty$.

    (ii) Is the limiting distribution in (i) concentrated on equilateral triangles and quadrangles?

    (iii) When $n$ is very large and $C_n'$ has a shape close to a (non-necessarily equilateral) triangle, what is the speed of growth of $C_n$ in various directions as a function of the angles of the triangle?

    (iv) Find the limiting distribution of the number of discs touching $\prt C_n$ as $n\to\infty$.

    (v) Prove a non-trivial lower bound (better than $\gtrsim n^{1/2}$) for the diameter of $C_n$.

    (vi) Can $C_n'$ converge to a polygon with obtuse angles (necessarily with at least 5 vertices)? 
\end{problem}

\begin{figure}
    \centering
    \begin{minipage}{0.33\textwidth}
        \centering
        \includegraphics[width=0.99\textwidth]{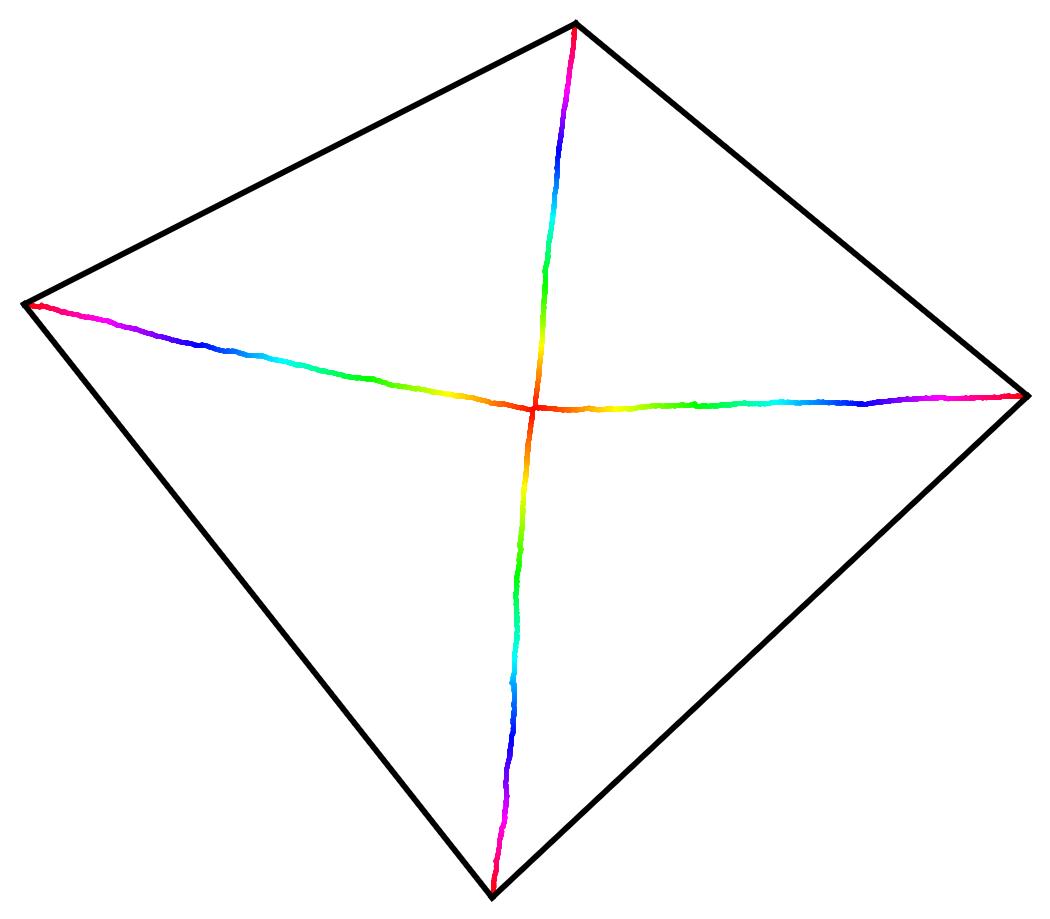}
    \end{minipage}\hfill
    \begin{minipage}{0.33\textwidth}
        \centering
        \includegraphics[width=0.99\textwidth]{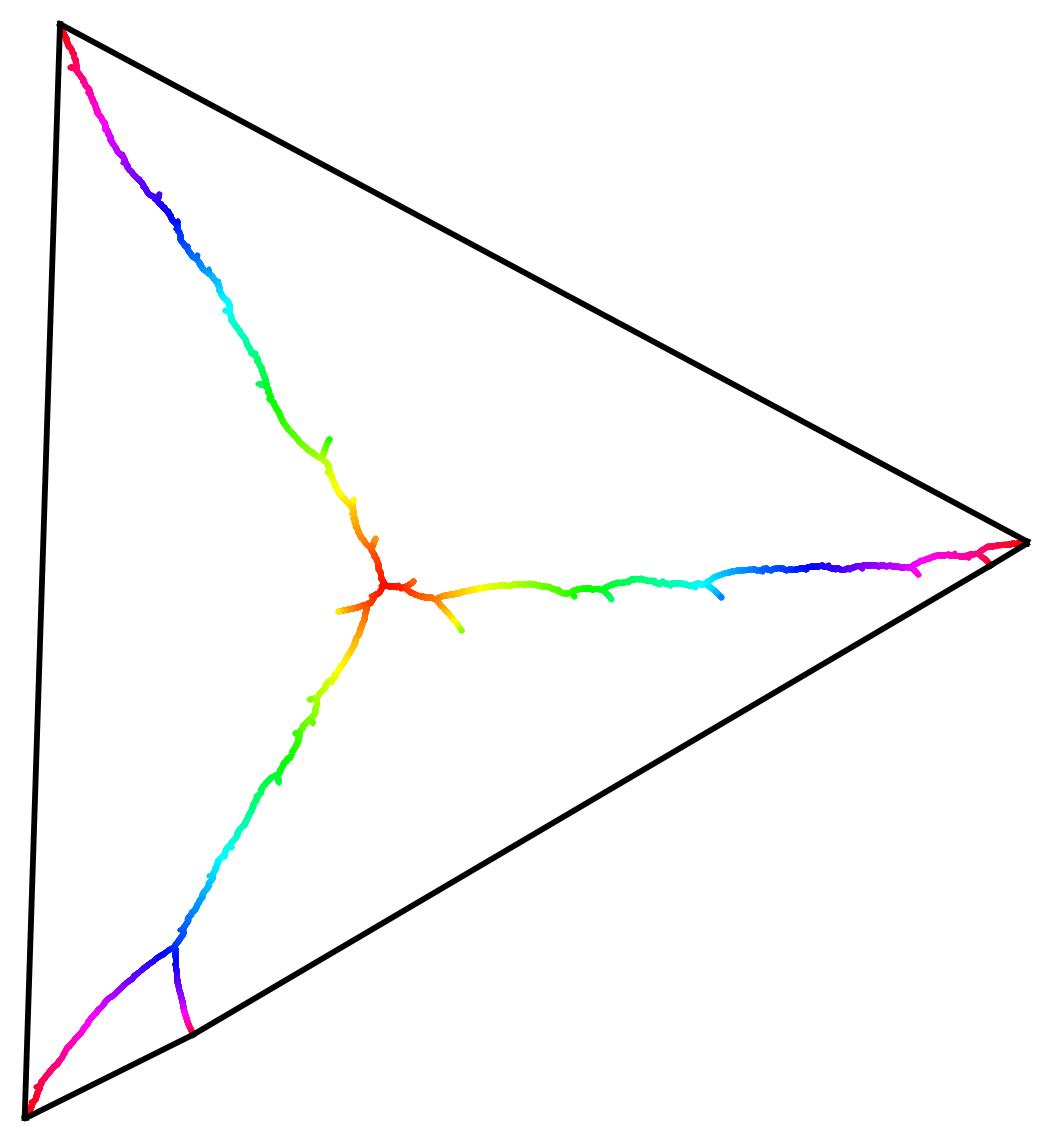} 
    \end{minipage}
    \begin{minipage}{0.33\textwidth}
        \centering
        \includegraphics[width=0.99\textwidth]{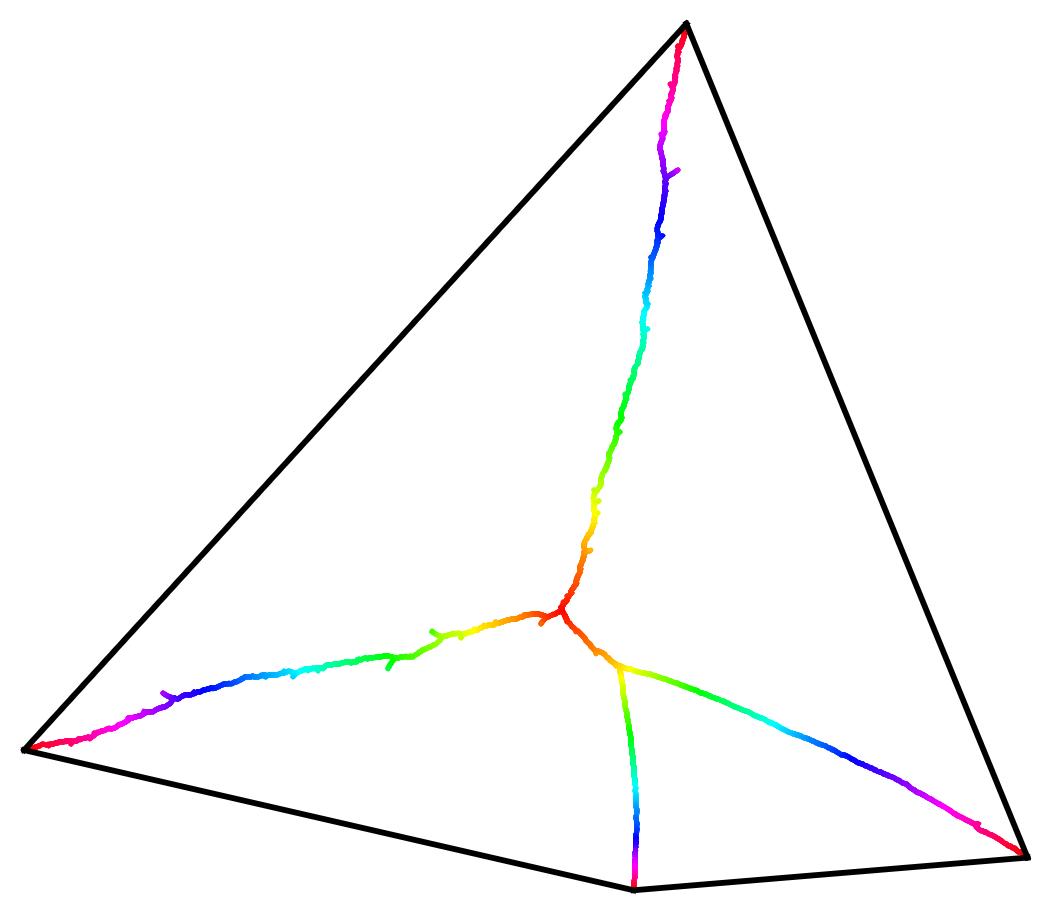}  
    \end{minipage}
    \\
\begin{minipage}{0.33\textwidth}
        \centering
        \includegraphics[width=0.99\textwidth]{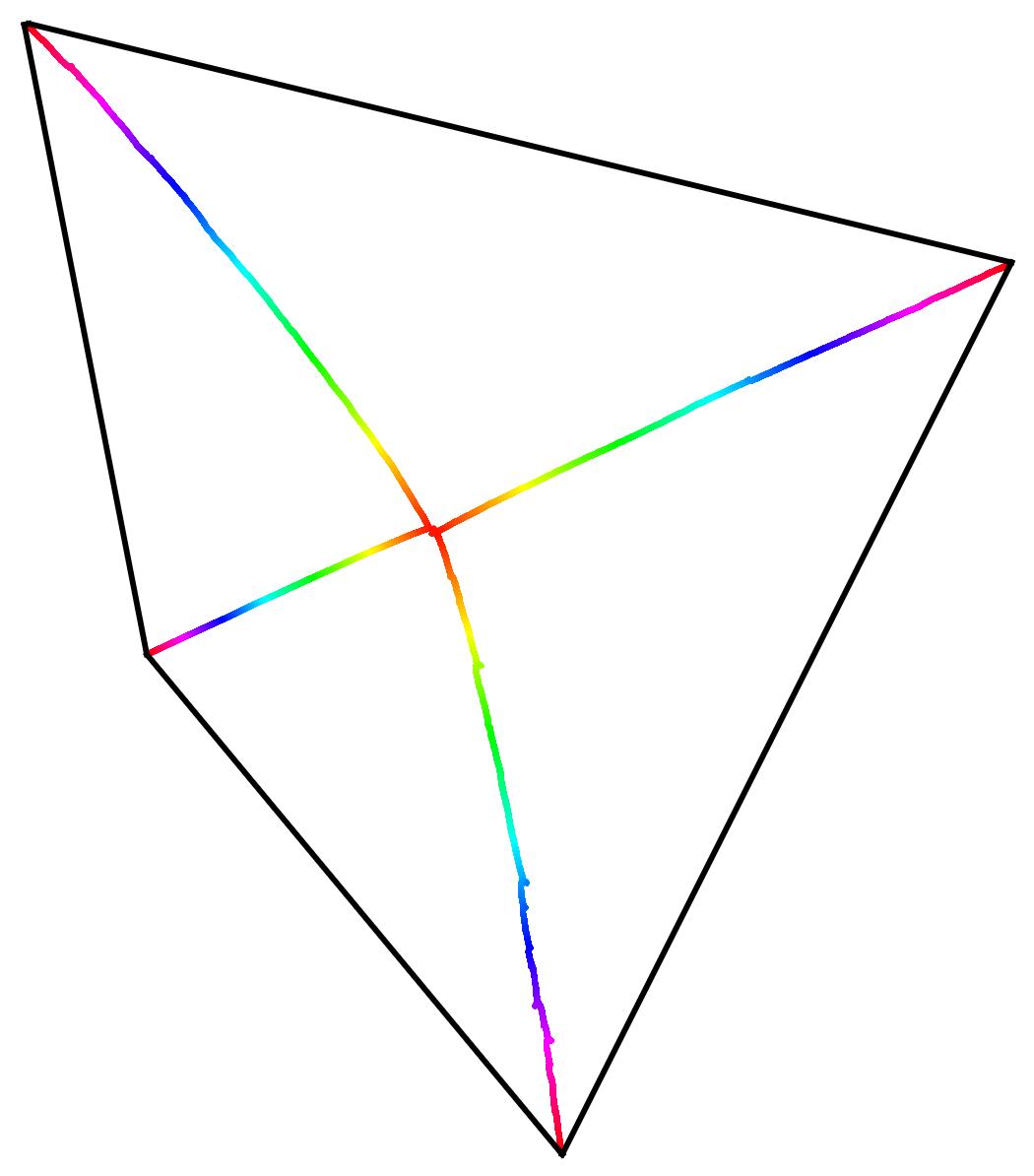}  
    \end{minipage}\hfill
    \begin{minipage}{0.33\textwidth}
        \centering
        \includegraphics[width=0.99\textwidth]{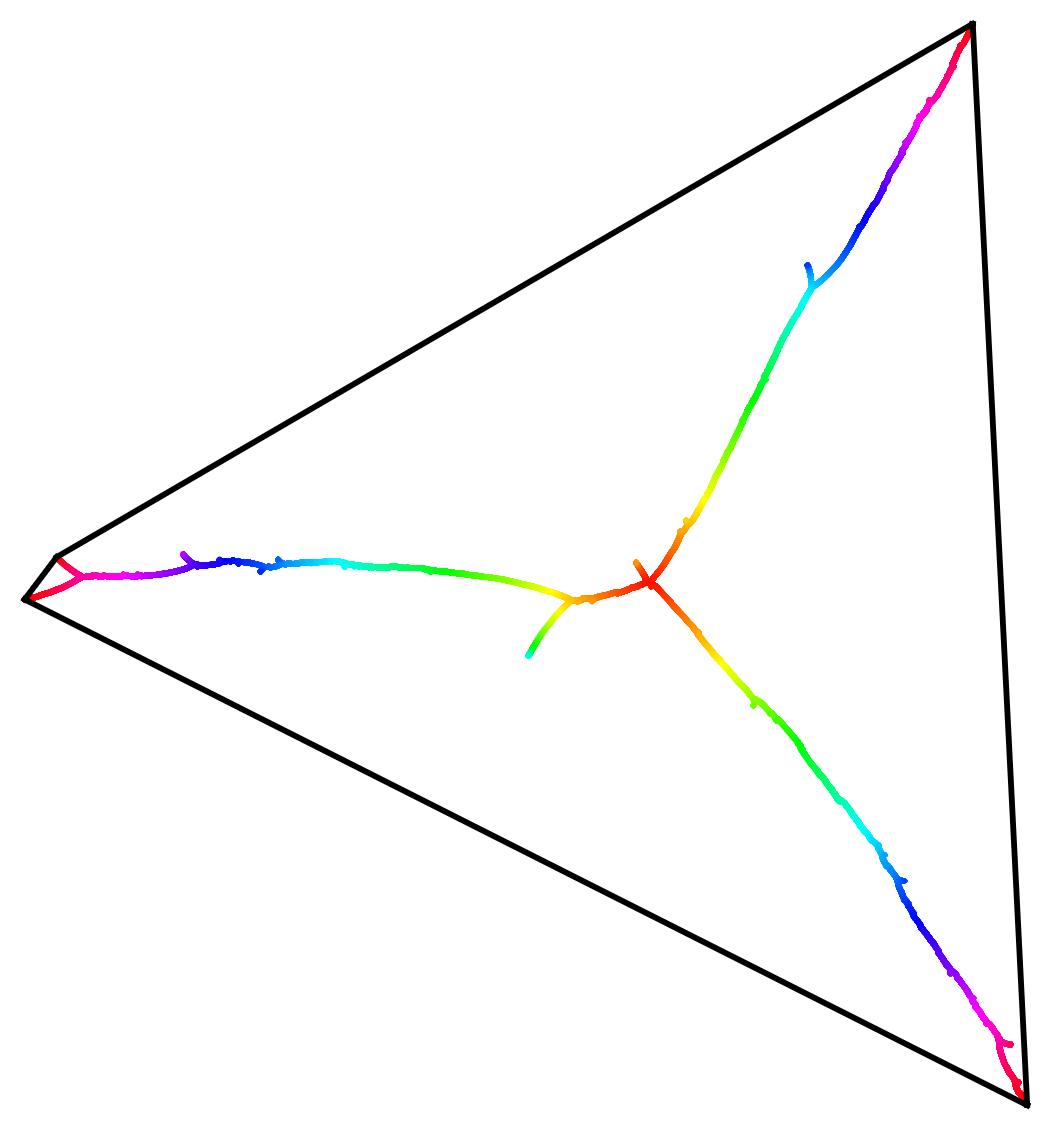}  
    \end{minipage}
    \begin{minipage}{0.33\textwidth}
        \centering
        \includegraphics[width=0.99\textwidth]{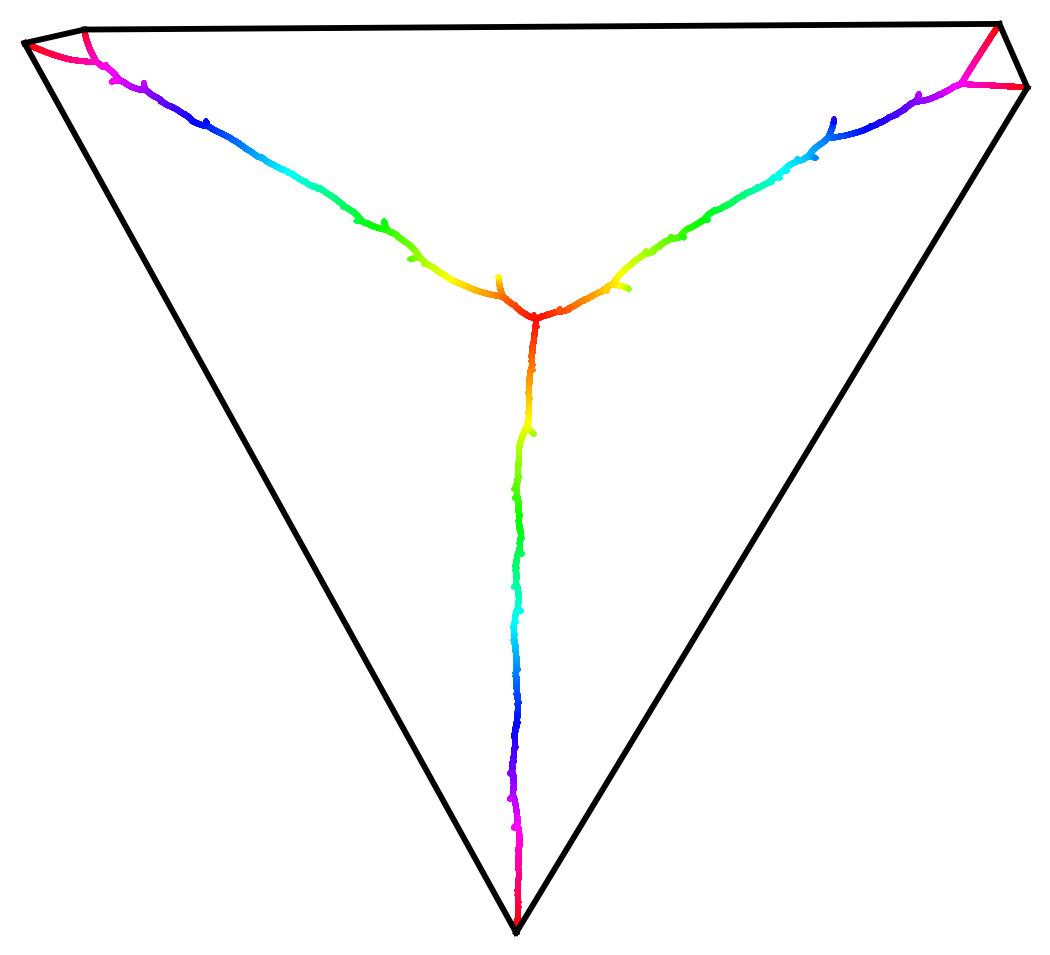}  
    \end{minipage}
    \\
\begin{minipage}{0.33\textwidth}
        \centering
        \includegraphics[width=0.99\textwidth]{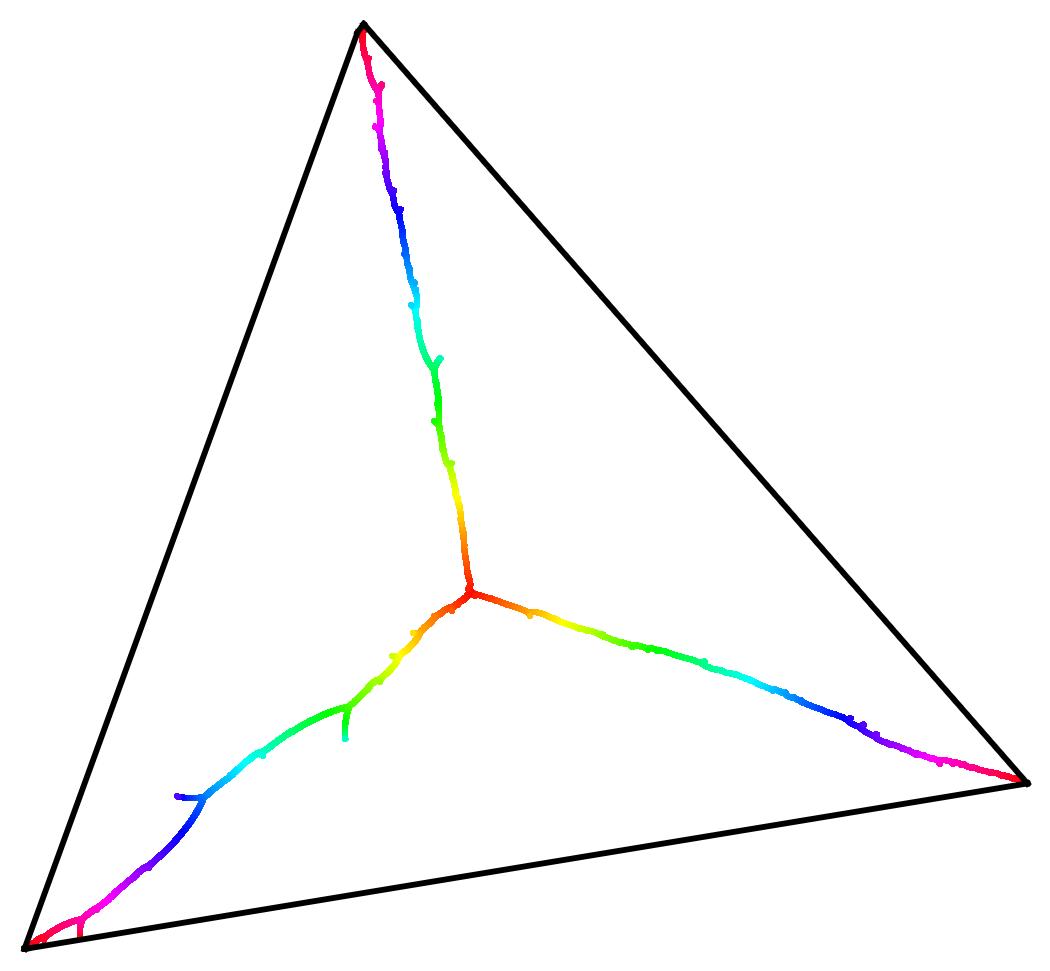}  
    \end{minipage}\hfill
    \begin{minipage}{0.33\textwidth}
        \centering
        \includegraphics[width=0.99\textwidth]{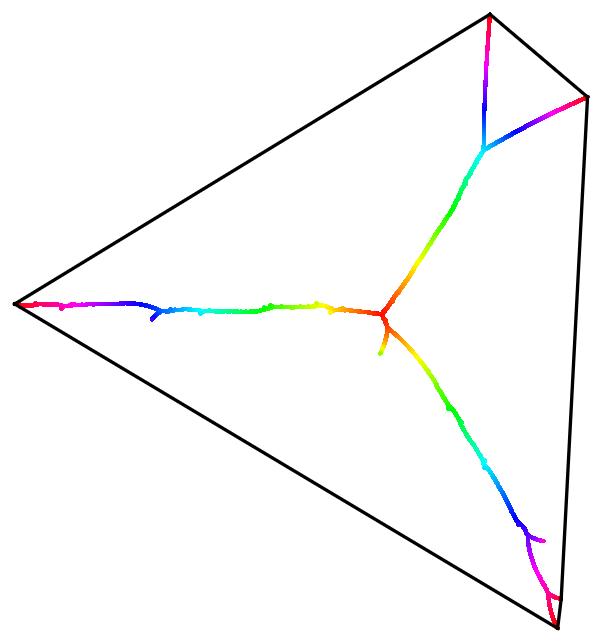}  
    \end{minipage}
    \begin{minipage}{0.33\textwidth}
        \centering
        \includegraphics[width=0.99\textwidth]{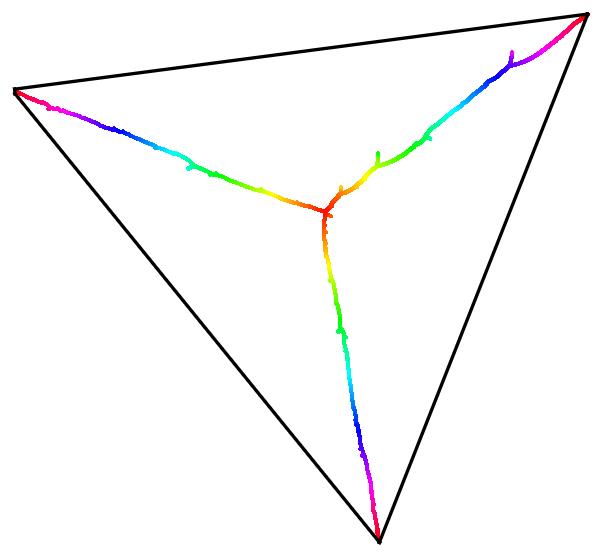}  
    \end{minipage}
    \caption{Nine simulations of clusters. Colors indicate the order of attachment. Boundaries of convex hulls are black. The number of discs in a cluster: $10^4$ in the top row, $10^5$ in the middle row, $10^6$ in the bottom row. } \label{fig_basic}
\end{figure}

\section{Directional growth}\label{o30.12}

Observations made in Section \ref{o30.2} give support to the following 
speculation. It appears that for large $n$, the convex hull $C_n$ is close to a polygonal line,
typically a triangle or quadrangle. Given a large $n$, for the next $k$ steps, if $k\ll n$, the edges of the polygonal
line will not substantially change their angles. For this reason, the evolution of the convex hull
might be decomposed into independent evolutions close to (approximate) vertices.

We will study a model of evolution at a single ``vertex.''
Suppose that $C_n$ is close to a triangle positioned so that two of the
edges of $\prt C_n$, say $I_1$ and $I_2$, are inclined at angles $\pi/2-\theta$ with the vertical
in the opposite directions (see Fig. \ref{fig1})
and form an approximate vertex.
We will assume that the other two vertices are ``at infinity,'' in the sense that
the angles formed by $I_1$ and $I_2$ with the vertical never change, i.e., they are always equal to $\pi/2-\theta$.

\begin{figure} [h]
\includegraphics[width=0.9\linewidth]{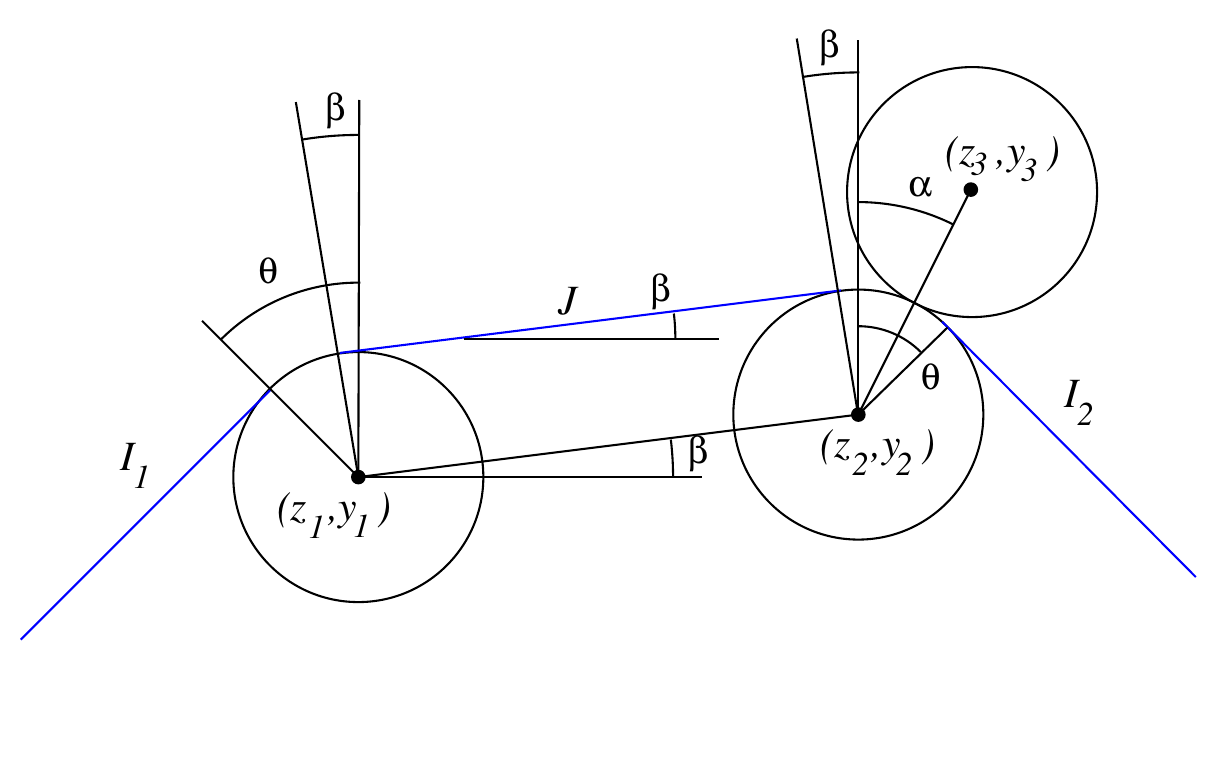}
\caption{ Blue lines are line segments in the convex hull. The ones on the left and right always form angles $\pi/2-\theta$ with the vertical.
The one in the middle forms a varying angle $\beta$ with the horizontal. A new disc has been attached to the right disc. The line passing through
the centers of the discs on the right forms an angle $\alpha$ with the vertical.}
\label{fig1}
\end{figure}

Suppose that there is only one extremal disc (in this neighborhood) and $\theta< \pi /4$. Then the angle between $I_1$ and $I_2$ is obtuse
and it is elementary to see that if we add discs to the cluster in this neighborhood then there will be always only one extremal disc in this area.

If there are two extremal discs then there will be a third edge between them, say $J$.
If $\theta < \pi/4$ then sooner or later, after some discs are added, only one extremal disc will remain.

We turn to the interesting case when $\theta > \pi/4$. Then the number of extremal discs is unlimited
but only the case of two extremal discs seems to be interesting. Our previous observations imply that as long as the angles between $I_1$ and $J$, and $J$ and $I_2$ are greater than $\pi/2$, there will be two extremal discs. It is easy to see that the distance between the centers of the two extremal discs 
has to grow. Hence, we could have a fork with two branches growing at comparable speeds. Binary forks are visible in most of our simulations.
It is clear from the simulations that one of the branches of a fork has to terminate. This is because the shorter branch grows slower and so its tip will disappear from the boundary of the convex hull.

We will try to determine the probability of the occurrence of macroscopic branches as a function of $\theta$.
If there is a fork, sub-forks are unlikely until it is very skewed so we will assume that forks occur sequentially. 
Let $T_k$ be the lifetime of the $k$-th fork.
The standard limit theorems for i.i.d. random variables show that macroscopic branches will occur if 
\begin{align}\label{o30.3}
    \P(T_k >n) > c n^{-\mu}
\end{align}
with $\mu<1$, and macroscopic branches will not occur if 
\begin{align}\label{o30.4}
    \P(T_k >n) < c n^{-\mu}
\end{align}
with $\mu>1$, for large $n$.

The next two sections are devoted to calculations trying to determine whether \eqref{o30.3}-\eqref{o30.4} are true and how $\mu$ depends on $\theta$. Of special interest is the case $\theta=\pi/3$ because it represents the neighborhood of a vertex in an equilateral triangle. Since equilateral triangles seem to be typical shapes for large clusters and our simulations show occasional macroscopic side branches, one would like to know whether these large side branches have a non-vanishing probability when $n$ goes to infinity.

\section{Calculation of one-step transition parameters}\label{o31.1}

Recall that the discs have radii $1/2$ and refer to Fig. \ref{fig1} for notation.

Let $(\calZ_n,\calY_n)$ denote the vector from the center of the left disc to the center of the right disc
after $n$ steps ($(\calZ_n,\calY_n)=(z_2-z_1, y_2-y_1)$ in Fig. \ref{fig1}).
Let $\calX_n=\calY_n/\calZ_n $ be the tangent of the angle between the line through the centers of the discs and the horizontal
($\calX_n= \tan\beta$ in Fig. \ref{fig1}).

\begin{theorem}\label{o30.10}
    We have
\begin{align}\label{o30.5}
     &\lim_{z\to\infty} z \E(\calX_{n+1}-\calX_n\mid \calX_n=x\geq 0,\calZ_n=z) 
     = \frac{x\cos\theta}{\theta},\\
     & \lim_{z\to\infty} \frac{z}{x} \E(\calX_{n+1}-\calX_n \mid \calX_n=x\geq 0,\calZ_n=z) 
     = \frac{\cos\theta}{\theta},\label{o30.6}\\
     &\lim_{z\to\infty} z^2\Var(\calX_{n+1}-\calX_n\mid \calX_n=x\geq 0,\calZ_n=z)\label{o30.7}\\
     &\qquad=\frac{1}{2 \theta ^2}\left( \left(x^2+1\right)\theta ^2-2 x^2 \cos ^2(\theta )-  \left(x^2-1\right) \theta\sin (\theta ) \cos (\theta )\right),\notag\\
     &\lim_{x\to0}\lim_{z\to\infty} z^2 \Var(\calX_{n+1}-\calX_n\mid \calX_n=x\geq 0,\calZ_n=z)
     =\frac{1}{2 \theta ^2}\left( \theta ^2+ \theta\sin (\theta ) \cos (\theta )\right).\label{o30.8}
\end{align}
\end{theorem}

\begin{remark}
    We have fully explicit formulas for $\E(\calX_{n+1}-\calX_n\mid \calX_n=x\geq 0,\calZ_n=z) $
    and $\E((\calX_{n+1}-\calX_n)^2\mid \calX_n=x\geq 0,\calZ_n=z) $, and, therefore, for
    $\Var(\Delta x\mid \calX_n=x\geq 0,\calZ_n=z)$, see \eqref{o30.8a} and \eqref{o30.9}. The formulas are so complicated that they do not seem to be useful in the non-asymptotic form.
\end{remark}

\begin{proof}[Proof of Theorem \ref{o30.10}]
Many calculations presented in this proof were performed with Mathematica.

If a disc is attached to the right disc at an angle $\alpha$ from the vertical line, with the positive angle to the right, then
\begin{align*}
    \calY_{n+1} &= \calY_n + \cos\alpha,\\
    \calZ_{n+1} & =\calZ_n + \sin \alpha ,\\
    \calX_{n+1}=\calY_{n+1}/\calZ_{n+1}& = ( \calY_{n} + \cos \alpha)/( \calZ_n + \sin\alpha)= ( \calX_n\calZ_n + \cos \alpha)/( \calZ_n + \sin\alpha).
\end{align*}
When a disc is attached to the left disc at an angle $\alpha$ from the vertical line, with the positive angle to the left,
then
\begin{align*}
    \calY_{n+1} &= \calY_n - \cos\alpha,\\
    \calZ_{n+1} & =\calZ_n + \sin \alpha ,\\
    \calX_{n+1}=\calY_{n+1}/\calZ_{n+1}& = ( \calY_{n} - \cos \alpha)/( \calZ_n + \sin\alpha)= ( \calX_n\calZ_n - \cos \alpha)/( \calZ_n + \sin\alpha).
\end{align*}

If $\calX_n=x\geq 0$ and $\calZ_n=z$,
the expected value of $\calX_{n+1}-\calX_n$  after one disc is attached is
\begin{align}\notag
 \E&(\calX_{n+1}-\calX_n \mid \calX_n=x\geq 0,\calZ_n=z)\\
 &= \frac{1}{2\theta}  \int_{-\arctan{x}}^\theta \left(\frac {xz + \cos \alpha}{ z + \sin\alpha} -x\right) d\alpha
+  \frac{1}{2\theta}  \int_{\arctan{x}}^\theta \left(\frac {xz - \cos \alpha}{ z + \sin\alpha} -x\right) d\alpha\notag\\
&=-\frac{1}{2 \theta } \Bigg(2\theta x-\frac{4 z x \arctan\left(\frac{z \tan \left(\frac{\theta }{2}\right)+1}{\sqrt{z^2-1}}\right)}{\sqrt{z^2-1}}\label{o30.8a}\\
&\qquad +\frac{2 z x \arctan\left(\frac{z \tan \left(\frac{1}{2} \arctan(x)\right)+1}{\sqrt{z^2-1}}\right)}{\sqrt{z^2-1}}+\frac{2 z x \arctan\left(\frac{1-z \tan \left(\frac{1}{2} \arctan(x)\right)}{\sqrt{z^2-1}}\right)}{\sqrt{z^2-1}}\notag\\
&\qquad -\log \left(z+\frac{x}{\sqrt{x^2+1}}\right)+\log \left(z-\frac{x}{\sqrt{x^2+1}}\right)\Bigg).\notag
\end{align}
We will examine the behavior of the above expression when $z\to \infty$.
The quantity goes to 0 like $1/z$ so we multiply consecutive lines by $z$ and take the limits as $z\to \infty$.
\begin{align*}
    &\lim_{z\to\infty}\left( -\frac{1}{2\theta} z \left(2\theta x-\frac{4 z x \arctan\left(\frac{z \tan \left(\frac{\theta }{2}\right)+1}{\sqrt{z^2-1}}\right)}{\sqrt{z^2-1}}\right)\right)
    = \frac{2 x \cos ^2\left(\frac{\theta }{2}\right)}{\theta }
    = \frac{1}{\theta} x (1+\cos\theta),\\
    &\lim_{z\to\infty}\left( -\frac{1}{2\theta} z 
    \left(\frac{2 z x \arctan\left(\frac{z \tan \left(\frac{1}{2} \arctan(x)\right)+1}{\sqrt{z^2-1}}\right)}{\sqrt{z^2-1}}
    +\frac{2 z x \arctan\left(\frac{1-z \tan \left(\frac{1}{2} \arctan(x)\right)}{\sqrt{z^2-1}}\right)}{\sqrt{z^2-1}}\right)\right)\\
    &\qquad= -\frac{x \left(\frac{1}{\sqrt{x^2+1}}+1\right)}{\theta },\\
      &\lim_{z\to\infty}\left( -\frac{1}{2\theta} z \left(-\log \left(z+\frac{x}{\sqrt{x^2+1}}\right)+\log \left(z-\frac{x}{\sqrt{x^2+1}}\right)\right)\right)
    = \frac{x}{\theta  \sqrt{x^2+1}}.
\end{align*}
Combining the above into one formula, we obtain
\begin{align*}
     &\lim_{z\to\infty} z \E(\calX_{n+1}-\calX_n\mid \calX_n=x\geq 0,\calZ_n=z) \\
     &\qquad= \frac{1}{\theta}
     \left( \frac{x}{ \sqrt{x^2+1}} -  x \left(\frac{1}{\sqrt{x^2+1}}+1\right) + x (1+\cos\theta)\right)
     = \frac{x\cos\theta}{\theta}.
     \end{align*}
This proves \eqref{o30.5} and \eqref{o30.6}.

We have
\begin{align}\notag
 f(x,z,\theta) &:= \E((\calX_{n+1}-\calX_n)^2\mid \calX_n=x\geq 0,\calZ_n=z)\\
 &= \frac{1}{2\theta}  \int_{-\arctan{x}}^\theta \left(\frac {xz + \cos \alpha}{ z + \sin\alpha} -x\right)^2 d\alpha
+  \frac{1}{2\theta}  \int_{\arctan{x}}^\theta \left(\frac {xz - \cos \alpha}{ z + \sin\alpha} -x\right)^2 d\alpha \notag \\
&=
\frac{1}{2 \theta }\Bigg(-\frac{4 z \left(z^2 \left(x^2-1\right)-2 x^2+1\right)  \arctan \left(\frac{z \tan \left(\frac{\theta }{2}\right)+1}{\sqrt{z^2-1}}\right)}{\left(z^2-1\right)^{3/2}} \label{o30.9} \\
&\quad+\frac{\cos (\theta )+z^2 \left(x^2-1\right) \cos (\theta )-2 \left(z^2-1\right) z x}{(z-1) (z+1) (\sin (\theta )+z)} \notag \\
&\quad+\frac{\cos (\theta )+z^2 \left(x^2-1\right) \cos (\theta )+2 \left(z^2-1\right) z x}{(z-1) (z+1) (\sin (\theta )+z)} \notag \\
&\quad+\frac{2 z \left(z^2 \left(x^2-1\right)-2 x^2+1\right)  \arctan \left(\frac{z \tan \left(\frac{1}{2}  \arctan (x)\right)+1}{\sqrt{z^2-1}}\right)}{\left(z^2-1\right)^{3/2}} \notag \\
&\quad +\frac{2 z \left(z^2 \left(x^2-1\right)-2 x^2+1\right)  \arctan \left(\frac{1-z \tan \left(\frac{1}{2}  \arctan (x)\right)}{\sqrt{z^2-1}}\right)}{\left(z^2-1\right)^{3/2}} \notag \\
&\quad-\frac{-2 z^3 x \sqrt{x^2+1}+z^2 \left(x^2-1\right)+2 z x \sqrt{x^2+1}+1}{(z-1) (z+1) \left(z \sqrt{x^2+1}-x\right)} \notag \\
&\quad -\frac{2 z^3 x \sqrt{x^2+1}+z^2 \left(x^2-1\right)-2 z x \sqrt{x^2+1}+1}{(z-1) (z+1) \left(z \sqrt{x^2+1}+x\right)} \notag \\
&\quad -2 x \log \left(z+\frac{x}{\sqrt{x^2+1}}\right)+2 x \log \left(z-\frac{x}{\sqrt{x^2+1}}\right)+2 \theta  (x-1) (x+1)\Bigg). \notag 
\end{align}

The Taylor series expansion of the function $s\to f(x,1/s,\theta)$ for $s$ around $0$ generated by Mathematica is
\begin{align*}
     &\arctan \left(\frac{\tan \left(\frac{1}{2}  \arctan (x)\right)}{\sqrt{\frac{1}{s^2}} s}\right) \left(\frac{1-x^2}{\left(\left(\frac{1}{s^2}\right)^{3/2} s^3\right) \theta }+\frac{\left(x^2+1\right) s^2}{\left(\left(2 \left(\frac{1}{s^2}\right)^{3/2}\right) s^3\right) \theta }+O\left(s^4\right)\right)\\
     &+ \arctan \left(\frac{\tan \left(\frac{\theta }{2}\right)}{\sqrt{\frac{1}{s^2}} s}\right) \left(-\frac{2 \left(x^2-1\right)}{\left(\left(\frac{1}{s^2}\right)^{3/2} s^3\right) \theta }+\frac{\left(x^2+1\right) s^2}{\left(\left(\frac{1}{s^2}\right)^{3/2} s^3\right) \theta }+O\left(s^4\right)\right)\\
     &+ \arctan \left(\frac{\tan \left(\frac{1}{2}  \arctan (x)\right)}{\sqrt{\frac{1}{s^2}} s}\right) \left(\frac{x^2-1}{\left(\left(\frac{1}{s^2}\right)^{3/2} s^3\right) \theta }+\frac{\left(-x^2-1\right) s^2}{\left(\left(2 \left(\frac{1}{s^2}\right)^{3/2}\right) s^3\right) \theta }+O\left(s^4\right)\right)\\
     &+(x-1) (x+1)\\
     &+\frac{1}{2 \theta }\Bigg(\Bigg(-\left(4 \left(x^2-1\right)\right) \cos ^2\left(\frac{\theta }{2}\right)+\left(4 \left(x^2-1\right)\right) \cos ^2\left(\frac{1}{2}  \arctan (x)\right)\\
     &\quad +\left(2 \left(x^2-1\right)\right) \cos (\theta ) +2 \sqrt{x^2+1}-\frac{4 x^2}{\sqrt{x^2+1}}\Bigg) s\Bigg)\\
     &+\frac{1}{2 \theta }\Bigg(\Bigg(-\frac{x \left(x^2-1\right)}{2 \left(x^2+1\right)}-\left(\left(2 \left(x^2-1\right)\right) \cos (\theta )\right) \sin (\theta )+\frac{1}{2} \left(x^2-1\right) \sin (2 \theta )\\
     &\quad-\frac{1}{2} \left(x^2-1\right) \cos \left(\frac{1}{2}  \arctan (x)\right) \left(\sin \left(\frac{1}{2}  \arctan (x)\right)-\sin \left(\frac{3}{2}  \arctan (x)\right)\right)\Bigg) s^2\Bigg)\\
     & +\frac{1}{2 \theta }\Bigg(\Bigg(2 \left(x^2-1\right) \cos (\theta ) \sin ^2(\theta )\\
     &\quad+\left((2 x) \sqrt{x^2+1}\right) \left(\frac{x^3}{\left(x^2+1\right)^2}+\frac{x}{x^2+1}\right)\\
     &\quad-\left((2 x) \sqrt{x^2+1}\right) \left(-\frac{x}{x^2+1}-\frac{x^3}{\left(x^2+1\right)^2}\right)\\
     &\quad-\left(2 \left(x^2-1\right)\right) \left(\frac{x^2}{\left(x^2+1\right)^{3/2}}+\frac{1}{\sqrt{x^2+1}}\right)\\
     &\quad+\left(2 \left(x^2-1\right)\right) \cos (\theta )+2 \cos (\theta )\\
     &\quad-\frac{2}{3} \Bigg(-\left(3 x^2\right) \cos ^2\left(\frac{\theta }{2}\right)+\cos (2 \theta ) \cos ^2\left(\frac{\theta }{2}\right)-3 \cos ^2\left(\frac{\theta }{2}\right)-\left(2 \cos ^2\left(\frac{\theta }{2}\right)\right) \cos (\theta )\\
     &\qquad+\left(\left(2 x^2\right) \cos ^2\left(\frac{\theta }{2}\right)\right) \cos (\theta )-\left(x^2 \cos ^2\left(\frac{\theta }{2}\right)\right) \cos (2 \theta )\Bigg)\\
     &\quad-\frac{1}{3 \sqrt{x^2+1}}\Bigg(\left(2 \cos ^2\left(\frac{1}{2}  \arctan (x)\right)\right) \Bigg(-2 x^2+\left(x^2 \sqrt{x^2+1}\right) \cos \left(2  \arctan (x)\right)\\
     &\qquad-\sqrt{x^2+1} \cos \left(2  \arctan (x)\right)+\left(3 x^2\right) \sqrt{x^2+1}+3 \sqrt{x^2+1}+2\Bigg)\Bigg)\\
     &\quad-\frac{4 x^2}{\sqrt{x^2+1}}-\frac{2}{\sqrt{x^2+1}}-\frac{4 x^4}{3 \left(x^2+1\right)^{3/2}}\Bigg) s^3\Bigg)+O\left(s^4\right).
\end{align*}
After simplification, 

(i) the coefficients of $s$ and lower powers of $s$ are $0$, and 

(ii) the coefficient of $s^2$ is the quantity on the right hand side of \eqref{o20.1} below.

Hence,
\begin{align}\label{o20.1}
     \lim_{z\to\infty} &z^2 \E((\calX_{n+1}-\calX_n)^2\mid \calX_n=x\geq 0,\calZ_n=z)\\
     &= \frac{ \theta  \left(x^2+1\right)-\left(x^2-1\right) \sin (\theta ) \cos (\theta )}{2 \theta }\notag\\
     &= \frac{1}{2}   \left(x^2+1\right)-\frac{1}{2\theta} \left(x^2-1\right) \sin (\theta ) \cos (\theta ),\notag\\
      \lim_{z\to\infty} z^2\Var(\Delta x) &= \frac{1}{2}   \left(x^2+1\right)-\frac{1}{2\theta} \left(x^2-1\right) \sin (\theta ) \cos (\theta )
     - \left(\frac{x\cos\theta}{\theta}\right)^2\notag\\
     &=\frac{1}{2 \theta ^2}\left( \left(x^2+1\right)\theta ^2-2 x^2 \cos ^2(\theta )-  \left(x^2-1\right) \theta\sin (\theta ) \cos (\theta )\right).\notag
     \end{align}
This yields \eqref{o30.7} and \eqref{o30.8}.
\end{proof}

We have
\begin{align*}
 \E(&\calZ_{n+1}-\calZ_n\mid \calX_n=x\geq 0,\calZ_n=z)\\
 &= \frac{1}{2\theta}  \int_{-\arctan{x}}^\theta  \sin\alpha d\alpha
+  \frac{1}{2\theta}  \int_{\arctan{x}}^\theta  \sin\alpha d\alpha\\
&=\frac{1}{2\theta}  (-\cos \alpha)\Big|_{-\arctan{x}}^\theta + \frac{1}{2\theta} (-\cos \alpha)\Big|_{\arctan{x}}^\theta\\
&= \frac{1}{\theta} (\cos(\arctan{x}) - \cos \theta)= \frac{1}{\theta} \left(\frac{1}{\sqrt{1+x^2}} - \cos \theta\right).
\end{align*}
Letting $x\to 0$, we obtain
\begin{align}\label{o27.2}
    \lim_{x\to0}  \E(\calZ_{n+1}-\calZ_n\mid \calX_n=x\geq 0,\calZ_n=z) = \frac{1}{\theta} \left(1 - \cos \theta\right).
\end{align}

In view of \eqref{o27.2}, we conjecture that on a large scale, as long as $|\calX_n|$
stays below $\eps$ for some very small $\eps>0$,
$\calZ_n$ is close to a linear deterministic function $n\to Z_n:=\frac{1}{\theta} \left(1 - \cos \theta\right)n$.

\begin{conjecture}\label{o30.11}
    The process $\calX_n$ behaves like a continuous-time process $X_t$ satisfying
    \begin{align}\label{a22.3}
dX_t = \frac{\sigma (\theta)} t dW_t + \frac{\mu(\theta) X_t} t dt,
\end{align}
where $W$ is Brownian motion and the parameters 
\begin{align}\label{o27.5}
    \mu&=\mu(\theta)=\frac{(\cos\theta)/\theta}{\left(1 - \cos \theta\right)/\theta}
    = \frac{\cos\theta}{1 - \cos \theta},\\
    \sigma&=\sigma(\theta) =
    \frac{\sqrt{\frac{1}{2 \theta ^2}\left( \theta ^2+ \theta\sin (\theta ) \cos (\theta )\right)}}
    {\left(1 - \cos \theta\right)/\theta} =
    \frac{\sqrt{\left( \theta ^2+ \theta\sin (\theta ) \cos (\theta )\right)/2}}
    {1 - \cos \theta},\label{o27.6}
\end{align}
are obtained by dividing the quantities in \eqref{o30.6} and \eqref{o30.8} by the quantity in \eqref{o27.2}.
\end{conjecture}

Since $\mu>0$, the process $X_t$ is an ``anti-Ornstein-Uhlenbeck process,''
except that the diffusion and drift coefficients of the standard
Ornstein-Uhlenbeck process do not depend on time.

\section{Escape probability estimates for the anti-Ornstein-Uhlenbeck process}\label{secOU}

Recall that we would like to know whether \eqref{o30.3}-\eqref{o30.4} hold true
and how $\mu$ depends on $\theta$. Our choice of the notation, the same letter $\mu$ in 
\eqref{o30.3}-\eqref{o30.4} and in \eqref{o27.5} is not coincidental, as our next theorem shows.
In view of Conjecture \ref{o30.11}, the next result sheds light on the rate of 
creation of side branches in disc clusters.

\begin{theorem}\label{o27.7}
    Suppose that $X_t$ satisfies \eqref{a22.3} and the initial condition is $X_1=0$. Recall $\mu$ is defined in \eqref{o27.5}, fix an arbitrary  $a>0$ and let $T=\inf\{t\geq 0: |X_t| = a\}$.
    There exist $0<c', c''<\infty$ depending on $\theta$, and $t_1>0$ such that for $t>t_1$,
    \begin{align}\label{o22.7}
c' t^{-\mu}\leq \P(T>t) \leq c'' t^{-\mu}.
\end{align}

\end{theorem}

\begin{remark}
    We note that, somewhat surprisingly, the bounds in \eqref{o22.7} appear to depend
    only on $\mu$. However, the formulas for constants $c'$ and $c''$ obtained in the proof
    of the theorem depend on $\mu$ and $\sigma$ defined in \eqref{o27.6}.
\end{remark}

\begin{problem}
Recall the process $\calX_n$ defined in Section \ref{o31.1}.  Fix  $a>0$ and
    let $T_\calX=\inf\{n\geq 0: |\calX_n| = a\}$. Is it true that
 for some $c_*>0$ depending on $\theta$ and sufficiently large $n$,
    \begin{align}\label{o31.2}
c_* n^{-\mu}\leq \P(T_\calX>n) \leq (1/c_*) n^{-\mu},
\end{align}
where $\mu$ is defined in \eqref{o27.5}?
\end{problem}

\begin{proof}[Proof of Theorem \ref{o27.7}]
Let $ \wt X_{t}  = X(e^t) $ so that $\wt X_0=0$. It is elementary to check that this deterministic time change
results in the following representation of $\wt X$,
\begin{align*}
d\wt X_{t} 
&= \sigma (\theta) e^{-t/2} d \wt W_{t} 
+ \mu(\theta) \wt X_{t} dt,
\end{align*}
where  $\wt W_t $ is a Brownian motion.

Let $f(x,t)=x e^{-\mu t}$ and $\wh X_t = e^{-\mu t} \wt X_t = f(\wt X_t,t)$. We have  $f_x(x,t) = e^{-\mu t}$, $f_{xx}(x,t)=0$, $f_t(x,t) = -\mu e^{-\mu t}x$. By the Ito formula,
\begin{align}\label{a22.4}
d \wh X_t &=e^{-t(\mu+1/2)} \sigma   d \wt W_{t} 
+e^{-\mu t} \mu \wt X_{t} dt
-\mu e^{-\mu t} \wt X_{t} dt\\
&= e^{-t(\mu+1/2)} \sigma   d \wt W_{t} .
\end{align}
This implies that $\wh X_t$ is a time-changed Brownian motion, with integrable quadratic variation over $[0,\infty)$, hence converging to a finite limit.
The process $\wh X_t$ is Gaussian with independent increments. Its mean is 0 and its variance is
\begin{align*}
\E \wh X_t^2 = \int_0^t \sigma^2 e^{-2(\mu+1/2) s} ds = \frac{\sigma^2}{2(\mu+1/2)} (1-e^{-2(\mu+1/2) t}).
\end{align*}
Suppose $B_t$ is  Brownian motion with $B_0=0$ and let
\begin{align}\label{a22.5}
Y_t := B\left( \frac{\sigma^2}{2(\mu+1/2)} (1-e^{-2(\mu+1/2) t})\right).
\end{align}
Then $\{Y_t, t\in[0,\sigma^2/(2(\mu+1/2))\}$
has the same distribution as $\{\wh X_t, t\in[0,\sigma^2/(2(\mu+1/2))\}$. 
Note that this representation of $\wh X$ uses the values of 
Brownian motion $B$  only on a finite time interval $[0,\sigma^2/(2(\mu+1/2))$.

We want to  estimate
\begin{align}\notag
    \P(T>t) &= \P(\inf\{s\geq 0: |X_s| = a\} >t)= \P(\inf\{s\geq 0: |\wt X_{ s}| = a\} >\log t)\\
    &= \P(\inf\{s\geq 0: |\wh X_{ s}| = a e^{-\mu s}\} >\log t)= \P(\inf\{s\geq 0: |Y_s| = a e^{-\mu s}\} >\log t)\notag\\
    &= \P\left(\inf\left\{s\geq 0: \left|B\left( \frac{\sigma^2}{2(\mu+1/2)} (1-e^{-2(\mu+1/2) s})\right)\right| = a e^{-\mu s}\right\} >\log t\right).
    \label{o22.6}
\end{align}

Let 
\begin{align}\label{o22.4}
  S= \inf\left\{s\geq 0: \left|B\left( \frac{\sigma^2}{2(\mu+1/2)} (1-e^{-2(\mu+1/2) s})\right)\right| = a e^{-\mu s}\right\} .
\end{align}

Fix an $\eps>0$ such that $1/2-2\eps(\mu+1/2)>0$.
Since $B$ is Brownian motion, standard estimates show that
there exist $c_1,c_2>0$ and $u_1$ such that for $u\geq u_1$,
\begin{align}\notag
\P\Bigg( &\left|  B\left( \frac{\sigma^2}{2(\mu+1/2)} (1-e^{-2(\mu+1/2) u})\right) - B\left(\frac{\sigma^2}{2(\mu+1/2)}\right)\right|\\
&\qquad \geq \left( \frac{\sigma^2}{2(\mu+1/2)} e^{-2(\mu+1/2) u}\right)^{1/2-\eps}\Bigg)
\leq c_1 \exp(-c_2 e^{4(\mu+1/2) u\eps}).\label{o22.5}
\end{align}

We use the assumption that $1/2-2\eps(\mu+1/2)>0$ to see that if
\begin{align*}
    u\geq u_2:= \frac{1}{1/2-2\eps(\mu+1/2)} \log \left(\frac{1}{a} \left( \frac{\sigma^2}{2(\mu+1/2)} \right)^{1/2-\eps} \right)
\end{align*}
then
\begin{align}\label{o22.3}
    a e^{-\mu u} + \left( \frac{\sigma^2}{2(\mu+1/2)} e^{-2(\mu+1/2) u}\right)^{1/2-\eps}
    < 2 a e^{-\mu u}.
\end{align}

In the following calculation, the second inequality follows from \eqref{o22.3}, and the equality follows from the definition of $S$ stated in  \eqref{o22.4}.
The second to last inequality follows from \eqref{o22.5}.
If $u_3\geq u\geq \max(u_1,u_2)$ then
\begin{align*}
    \P&\left(\left|B\left(\frac{\sigma^2}{2(\mu+1/2)}\right)\right| \leq 2 a e^{-\mu u} \mid S=u_3\right)\\
    &\geq\P\left(\left|B\left(\frac{\sigma^2}{2(\mu+1/2)}\right)\right| \leq 2 a e^{-\mu S} \mid S=u_3\right)\\
&\geq \P\Bigg( \left| a e^{-\mu S} - B\left(\frac{\sigma^2}{2(\mu+1/2)}\right)\right|
\leq \left( \frac{\sigma^2}{2(\mu+1/2)} e^{-2(\mu+1/2) u}\right)^{1/2-\eps}\mid S=u_3\Bigg)\\
&=\P\Bigg( \left|  B\left( \frac{\sigma^2}{2(\mu+1/2)} (1-e^{-2(\mu+1/2) S})\right) - B\left(\frac{\sigma^2}{2(\mu+1/2)}\right)\right|\\
&\qquad \leq \left( \frac{\sigma^2}{2(\mu+1/2)} e^{-2(\mu+1/2) S}\right)^{1/2-\eps}\mid S=u_3\Bigg)\\
&   \geq 1 - c_1 \exp(-c_2 e^{4(\mu+1/2) S\eps})\geq 1 - c_1 \exp(-c_2 e^{4(\mu+1/2) u\eps}).
\end{align*}
This implies for $ u\geq \max(u_1,u_2)$,
\begin{align*}
    \P&\left(\left|B\left(\frac{\sigma^2}{2(\mu+1/2)}\right)\right| \leq 2 a e^{-\mu u} \right)\\
    &\geq \P(S\geq u) \P\left(\left|B\left(\frac{\sigma^2}{2(\mu+1/2)}\right)\right| \leq 2 a e^{-\mu u} \mid S\geq u\right)\\
    & \geq \P(S\geq u) (1 - c_1 \exp(-c_2 e^{4(\mu+1/2) u\eps})).
\end{align*}
For some constant $c_3$ depending on $a, \mu$ and $\sigma$, and $ u\geq \max(u_1,u_2)$,
\begin{align*}
    \P\left(\left|B\left(\frac{\sigma^2}{2(\mu+1/2)}\right)\right| \leq 2 a e^{-\mu u} \right)
    \leq c_3 e^{-\mu u},
\end{align*}
so the previous formula yields
\begin{align*}
    \P(S\geq u) \leq \frac{c_3 e^{-\mu u}}{1 - c_1 \exp(-c_2 e^{4(\mu+1/2) u\eps})}.
\end{align*}
This, \eqref{o22.6} and \eqref{o22.4}
imply that for some $c_4$ and all $t\geq e^{\max(u_1,u_2)}$,
\begin{align*}
 \P(T>t)=   \P(S\geq \log t) \leq \frac{c_3 t^{-\mu }}{1 - c_1 \exp(-c_2 t^{4(\mu+1/2) \eps})} \leq c_4 t^{-\mu}.
\end{align*}
This proves the upper bound in \eqref{o22.7}.

\medskip

To simplify notation,  write $\alpha=\mu/(2(\mu+1/2))$, $ c_* =  \frac{\sigma^2}{2(\mu+1/2)}$,
and $\gamma=a/c_*^\alpha$.
If we let $w = e^{-2(\mu+1/2) s}$ and $v=c_* w$ then $e^{-\mu s} = w^{\mu/(2(\mu+1/2))} = w^\alpha$ and for $u>0$,
\begin{align}\label{o27.4}
     \P&\left(\left|B\left( \frac{\sigma^2}{2(\mu+1/2)} (1-e^{-2(\mu+1/2) s})\right)\right| \leq a e^{-\mu s},s\in[0,u)\right)\\
     &=\P\left(\left|B\left( c_* -c_*w)\right)\right| \leq a w^\alpha,w\in[e^{-2(\mu+1/2) u},1)\right)\notag\\
     &=\P\left(\left|B\left( c_*-v\right)\right| \leq \gamma v^\alpha,v\in[c_*e^{-2(\mu+1/2) u},c_*)\right).\notag
\end{align}

Note that $\alpha=\mu/(2(\mu+1/2))<1/2$ for $\mu>0$. 
Fix a $\beta\in (\alpha, 1/2)$.

Fix an arbitrary $\delta>0$. By the local LIL for Brownian motion,
there exists $t_1>0$ such that if $t_2\in(0,t_1]$ and $W$ is Brownian motion starting from $W_0=0$ then
\begin{align*}
    \P(|W_t| < \gamma t^\beta, 0< t < t_2) > 1-\delta.
\end{align*}
Consider  $t_2\in(0,t_1]$. Let $A$ be the set of $x$ such that
\begin{align*}
    \P(|W_t| < \gamma t^\beta, 0< t < t_2 \mid W_{t_2}=x) \leq 1-2\delta.
\end{align*}
Let $p=\P(W_{t_2}\in A)$. Then
\begin{align*}
    p (1-2\delta) + (1-p)\cdot 1 
    &\geq \P(W_{t_2}\in A) \P(|W_t| < \gamma t^\beta, 0< t < t_2 \mid W_{t_2}=x)\\
    &\quad + \P(W_{t_2}\in A^c) \P(|W_t| < \gamma t^\beta, 0< t < t_2 \mid W_{t_2}=x)\\
    &=  \P(|W_t| < \gamma t^\beta, 0< t < t_2)> 1-\delta.
\end{align*}
This implies that $p< 1/2$ and, therefore, $\P(W_{t_2}\in A^c)>1/2$. 
Since the distribution of $W_{t_2}$ is centered normal with variance $t_2$,
$\P(|W_{t_2}| \leq 0.5 \sqrt{t_2 }) \approx 0.38$. Hence, using symmetry,
we conclude that there must be $x_1\in A^c$ with $x_1\geq  \sqrt{t_2 }/2$.
We have 
\begin{align}\label{o23.10}
 \P(W_t < \gamma t^\beta, 0< t < t_2 \mid W_{t_2}=x_1) \geq
    \P(|W_t| < \gamma t^\beta, 0< t < t_2 \mid W_{t_2}=x_1) \geq 1-2\delta.
\end{align}
Suppose that $x_2\in[-x_1,x_1]$.
A Brownian bridge from $(0,0)$ to $(t_2, x_2)$
has the same distribution as a bridge from $(0,0)$ to $(t_2, x_1)$
plus a linear function $t\to (x_2-x_1) t/t_2$. This and \eqref{o23.10} imply that for $x_2\in[-x_1,x_1]$,
\begin{align*}
 \P(W_t < \gamma t^\beta, 0< t < t_2 \mid W_{t_2}=x_2)  \geq 1-2\delta.
\end{align*}
By symmetry,
\begin{align*}
 \P(W_t >- \gamma t^\beta, 0< t < t_2 \mid W_{t_2}=x_2)  \geq 1-2\delta,
\end{align*}
so, for $|x_2| \leq \sqrt{t_2 }/2$,
\begin{align}\label{o25.2}
 \P(|W_t| < \gamma t^\beta, 0< t < t_2 \mid W_{t_2}=x_2)  \geq 1-4\delta.
\end{align}

Let $f(t)=\gamma t^\alpha$ and $g(t) = \gamma t^\beta$.
We have $f'(t)=\gamma \alpha t^{\alpha-1}$ and $g'(t) =\gamma \beta t^{\beta-1}$
so $f(t)-g(t)$ is increasing on the interval $(0,t_*)$, where $t_*$
satisfies
\begin{align*}
    &\alpha t_*^{\alpha-1}  =\beta t_*^{\beta-1},\\
    &t_*^{\alpha-\beta} = \beta/\alpha,\\
    &t_*  = (\beta/\alpha)^{1/(\alpha-\beta)}.
\end{align*}

We will argue that if  $s\in(0,t_*/2)$, $t\in(s, t_*)$ and $x \leq \gamma s^\alpha -\gamma s^\beta$ then
\begin{align}\label{o25.1}
    \gamma(t-s)^\beta +x \leq \gamma t^\alpha.
\end{align}
It will suffice to prove that 
\begin{align}\label{o23.4}
    \gamma(t-s)^\beta +\gamma s^\alpha -\gamma s^\beta \leq \gamma t^\alpha, \qquad s\in(0,t_*/2), t\in(s, t_*).
\end{align}
The above is equivalent to 
\begin{align}\label{o23.5}
    (t-s)^\beta + s^\alpha - s^\beta \leq  t^\alpha, \qquad s\in(0,t_*/2), t\in(s, t_*).
\end{align}
We have
\begin{align}\label{o23.1}
    t^\beta + s^\alpha - s^\beta \leq  t^\alpha
\end{align}
for $t=2s$ because the function $t\to t^\alpha-t^\beta$ is increasing on $(0,t_*)$.
For the same reason, \eqref{o23.1} holds for all $t\in[2s,t_*]$.
This implies that 
\begin{align}\label{o23.2}
    (t-s)^\beta + s^\alpha - s^\beta \leq  t^\alpha
\end{align}
for $t\in[2s,t_*]$. For $t\in[s,2s]$, $(t-s)^\beta  - s^\beta\leq 0$, so
\begin{align*}
    (t-s)^\beta + s^\alpha - s^\beta \leq s^\alpha \leq t^\alpha.
\end{align*}
This and \eqref{o23.2} imply \eqref{o23.5} and, therefore, \eqref{o23.4} and \eqref{o25.1}.

Let $t_3 = \min(c_*, t_1, t_*)$.
Note that
\begin{align}\label{o25.4}
   p_1:= \P(|B(c_*-t_3)| \leq \sqrt{t_3}/4, |B(c_*-v)| \leq \gamma v^\alpha, v\in[t_3,c_*]  \mid B(0) = 0)
\end{align}
is strictly positive.

Let $t_4 >0$ be so small that if $0< s \leq t_4$, $|y|\leq\gamma(s^\alpha - s^\beta)$ and $|x| \leq \sqrt{t_3}/4$ then
\begin{align}\label{o25.3}
    |x-y| \leq \sqrt{t_3-s}/2.
\end{align}
If $s\in(0,\min(t_4,t_3/2))$,
it is easy to see that there exists $c_5>0$ such that
if $|x| \leq \sqrt{t_3}/4$ then
\begin{align}\label{o25.5}
    \P(|B(c_*- s)| \leq  \gamma(s^\alpha - s^\beta)  \mid B(c_*-t_3) = x)
    \geq c_5 s^\alpha.
\end{align}
If $|x| \leq \sqrt{t_3}/4$ and $|y|\leq\gamma(s^\alpha - s^\beta)$ then by \eqref{o25.1},
\begin{align*}
    \P&(|B(c_*- v)| \leq  \gamma v^\alpha, v\in[s,t_3]   \mid B(c_*-t_3) = x, B(c_*- s) = y )\\
    &\geq \P(|B(c_*- v)| \leq  \gamma(v-s)^\beta +|y| , v\in[s,t_3]  \mid B(c_*-t_3) = x, B(c_*- s) = y )\\
    &\geq \P(|B(c_*- v)-y| \leq  \gamma(v-s)^\beta  , v\in[s,t_3]  \mid B(c_*-t_3) = x, B(c_*- s) = y ).
\end{align*}
By time-reversibility and shift-invariance of Brownian bridges and \eqref{o25.2},
\begin{align*}
    \P&(|B(c_*- v)-y| \leq  \gamma(v-s)^\beta  , v\in[s,t_3]  \mid B(c_*-t_3) = x, B(c_*- s) = y )\\
    &=\P(|W(v)-y| \leq  \gamma(v-s)^\beta  , v\in[s,t_3]  \mid W(t_3) = x, W( s) = y )\\
    &= \P(|W_t| < \gamma t^\beta, 0< t < t_3-s \mid W_0=0, W_{t_3-s}=x-y)  \geq 1-4\delta.
\end{align*}
The last inequality follows from \eqref{o25.2} because $t_3-s \leq t_1$ and, by \eqref{o25.3}, $|x-y| \leq \sqrt{t_3-s}/2$.
The last two displayed formulas yield for $|x| \leq \sqrt{t_3}/4$ and $|y|\leq\gamma(s^\alpha - s^\beta)$,
\begin{align*}
    \P&(|B(c_*- v)| \leq  \gamma v^\alpha, v\in[s,t_3]   \mid B(c_*-t_3) = x, B(c_*- s) = y )\geq 1-4\delta.
\end{align*}
This and \eqref{o25.5} give  for $|x| \leq \sqrt{t_3}/4$,
\begin{align*}
    \P&(|B(c_*- v)| \leq  \gamma v^\alpha, v\in[s,t_3]   \mid B(c_*-t_3) = x )\geq( 1-4\delta) c_5 s^\alpha.
\end{align*}
We combine this with \eqref{o25.4} to obtain 
\begin{align*}
    \P&(|B(c_*- v)| \leq  \gamma v^\alpha, v\in[s,c_*]   \mid  B(0) = 0 )\geq p_1( 1-4\delta) c_5 s^\alpha.
\end{align*}
This,  \eqref{o22.6} and \eqref{o27.4} give
\begin{align}\notag
    \P(T>t) 
    &= \P\left(\inf\left\{s\geq 0: \left|B\left( \frac{\sigma^2}{2(\mu+1/2)} (1-e^{-2(\mu+1/2) s})\right)\right| = a e^{-\mu s}\right\} >\log t\right)\\
     &=\P\left(\left|B\left( c_*-v\right)\right| \leq \gamma v^\alpha,v\in[c_*e^{-2(\mu+1/2) \log t},c_*)\right)\label{o31.3}\\
    &= \P(|B(c_*- v)| \leq  \gamma v^\alpha, v\in[c_*t^{-2(\mu+1/2) },c_*]   \mid  B(0) = 0 )\notag\\
    &\geq p_1( 1-4\delta) c_5 (c_*t^{-2(\mu+1/2) })^\alpha = c_6 t^{-\mu}.\notag
\end{align}
This completes the proof of the lower bound in \eqref{o22.7}.
\end{proof}

\section{Self-organized criticality}\label{n1.5}

We will compare our rigorous results with the simulations presented in
Figs. \ref{fig_branch_low}, \ref{fig_branch_high} and \ref{fig_branch10_6}. The figures contain simulation results for the model introduced 
in Section \ref{o30.12}, with all clusters starting with a single extremal disc. The direction of growth is vertical (up) in Fig. \ref{fig1} but it is horizontal (right) in 
Figs. \ref{fig_branch_low}, \ref{fig_branch_high} and \ref{fig_branch10_6}, for typographical reasons. In other words, the clusters were turned 90 degrees clockwise in Figs. \ref{fig_branch_low}, \ref{fig_branch_high} and \ref{fig_branch10_6}.

Theorem \ref{o27.7} suggests that the bounds in \eqref{o31.2} hold true. Then the argument given in the paragraph containing \eqref{o30.3}-\eqref{o30.4} suggests that
macroscopic branches may grow when $\mu< 1$, i.e., when $(\cos\theta)/(1-\cos\theta)<1$, which is equivalent to $\theta> \pi/3$. 
Conversely, we expect not to see macroscopic branches when $\mu> 1$, i.e., when  $\theta> \pi/3$. 
Simulation results presented in Figs. \ref{fig_branch_low}, \ref{fig_branch_high} and \ref{fig_branch10_6} seem to support these claims, although the transition is not very sharp and seems to occur at a value slightly lower than $\pi/3$.

If $\pi/3$ is indeed the critical value of $\theta$, then Problem \ref{o30.2a} (ii), one of the most intriguing open problems, remains wide open because analysis of models at the critical value is particularly hard. On the other hand, simulations presented in Figs. \ref{fig_branch_low}, \ref{fig_branch_high} and \ref{fig_branch10_6} indicate that the critical value of $\theta$ is strictly less than $\pi/3$. Every row in Fig. \ref{fig_basic} contains convex hulls that are not triangular. Different rows represent different numbers of discs in the cluster: $10^4$, $10^5$ and $10^6$, so one may expect that macroscopic side branches occur on every scale and the probability of a non-triangular shape of $C_n$ does not go to 0 when $n\to\infty$.

The branches in our model are analogues of avalanches in sandpile SOC models, and fires in forest fire SOC models. In all of these cases, a single ``event'' (a ``relaxation event'') prevents events in the near future, at least on some scales. The size of branches in our model has power scaling, like events in other popular SOC models. 

\begin{figure}
  \centering
  \subfloat{\includegraphics[width=0.7\textwidth]{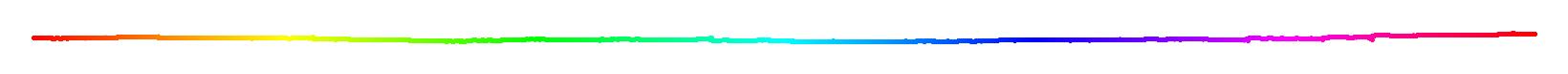} } \\
  \subfloat{\includegraphics[width=0.7\textwidth]{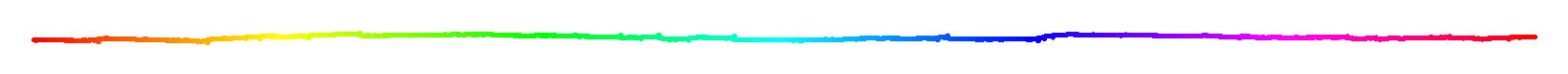} }\\
  \subfloat{\includegraphics[width=0.7\textwidth]{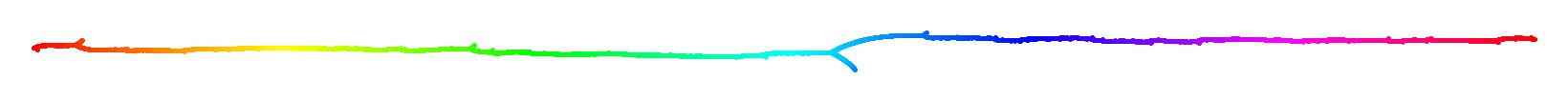} }\\
  \subfloat{\includegraphics[width=0.7\textwidth]{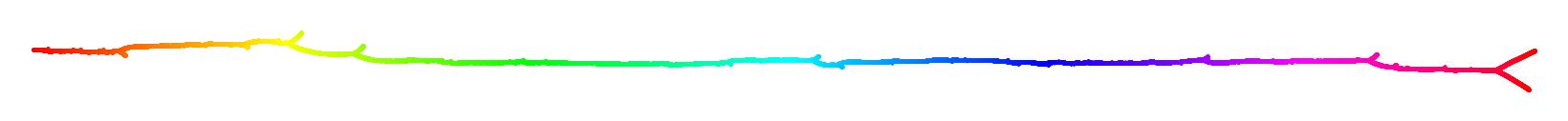} } \\
  \subfloat{\includegraphics[width=0.7\textwidth]{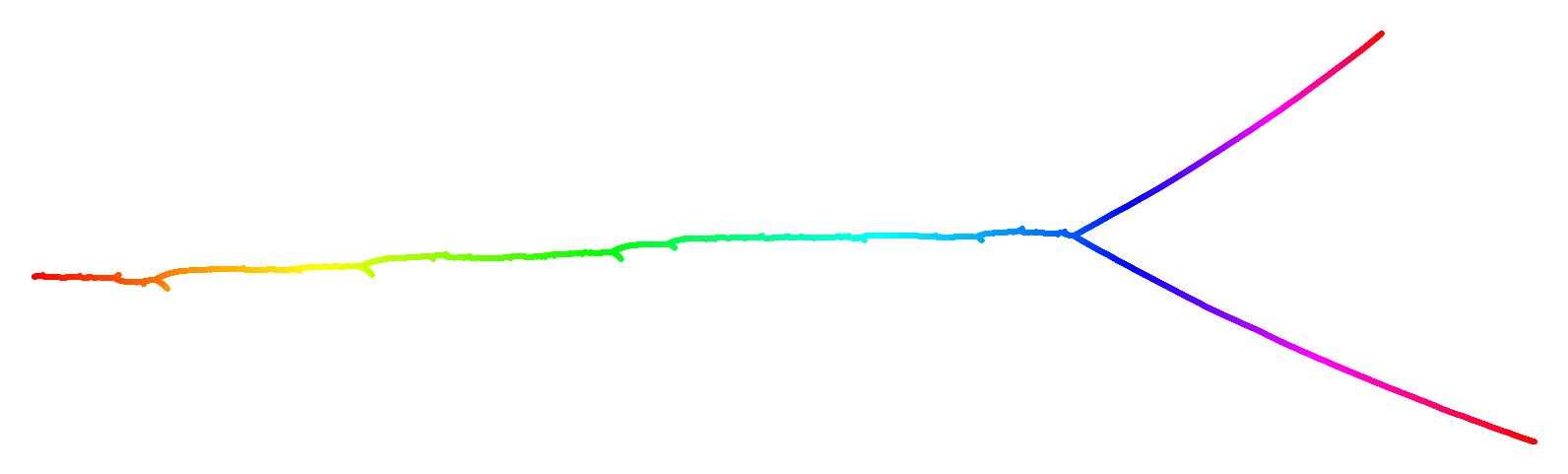} }\\
  \subfloat{\includegraphics[width=0.7\textwidth]{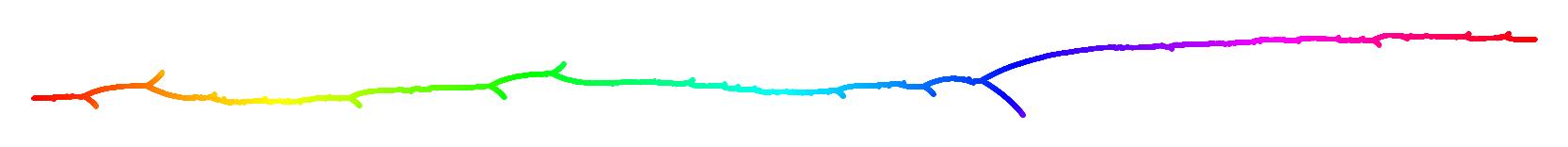} }\\
  \subfloat{\includegraphics[width=0.7\textwidth]{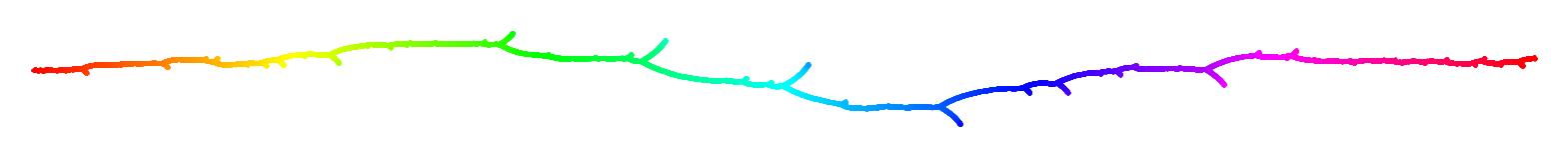} }\\
  \subfloat{\includegraphics[width=0.7\textwidth]{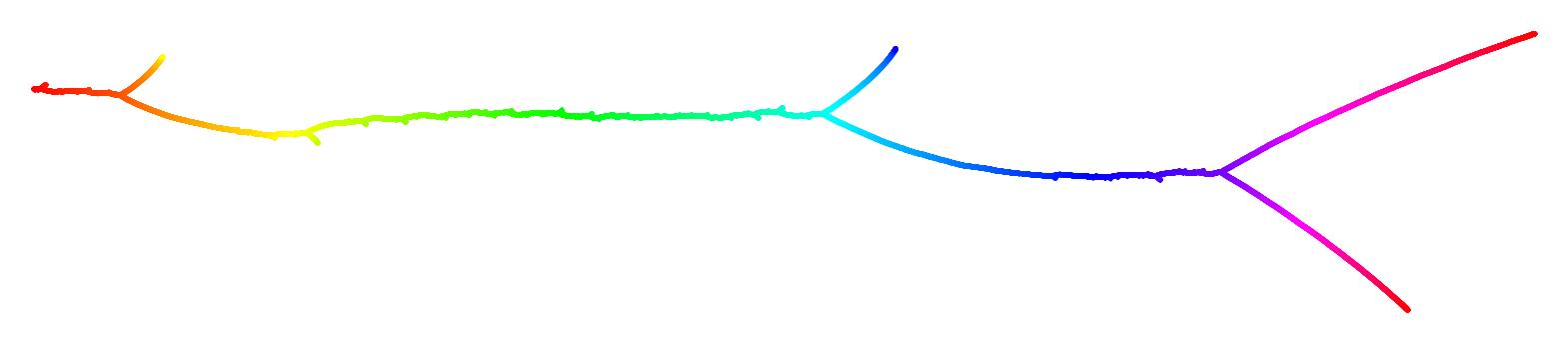} }
  \caption{Simulations of branching intensity. Each cluster contains $10^5$ discs. Colors indicate the order of attachment. The clusters represent different angles $\theta$, starting from the top: $0.3\pi$,  $0.305\pi$,  $0.31\pi$,  $0.315\pi$, $0.32\pi$, $0.325\pi$, $0.33\pi$, $0.335\pi$.} \label{fig_branch_low}
\end{figure}

\begin{figure}
  \centering
  \subfloat{\includegraphics[width=0.4\textwidth]{branchPi0_33_10_5.jpg} }\\
  \subfloat{\includegraphics[width=0.4\textwidth]{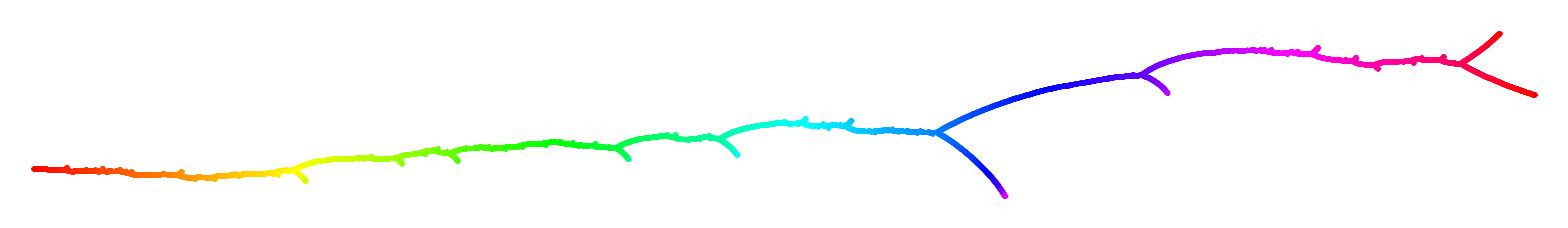} }\\
  \subfloat{\includegraphics[width=0.4\textwidth]{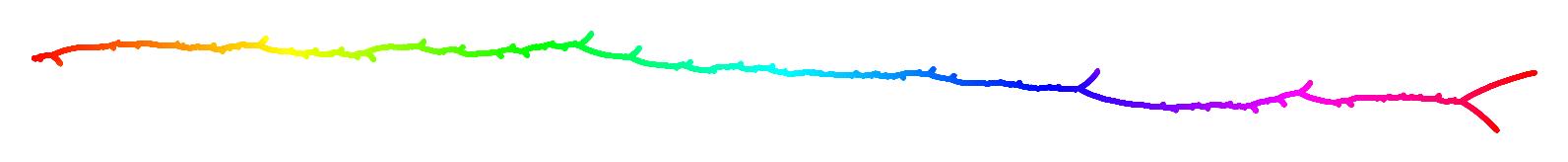} }\\
  \subfloat{\includegraphics[width=0.4\textwidth]{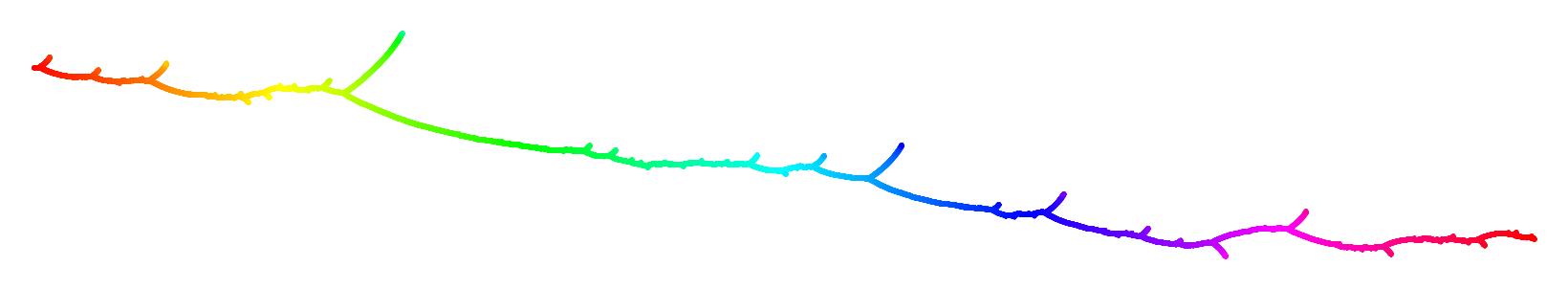} }\\
  \subfloat{\includegraphics[width=0.4\textwidth]{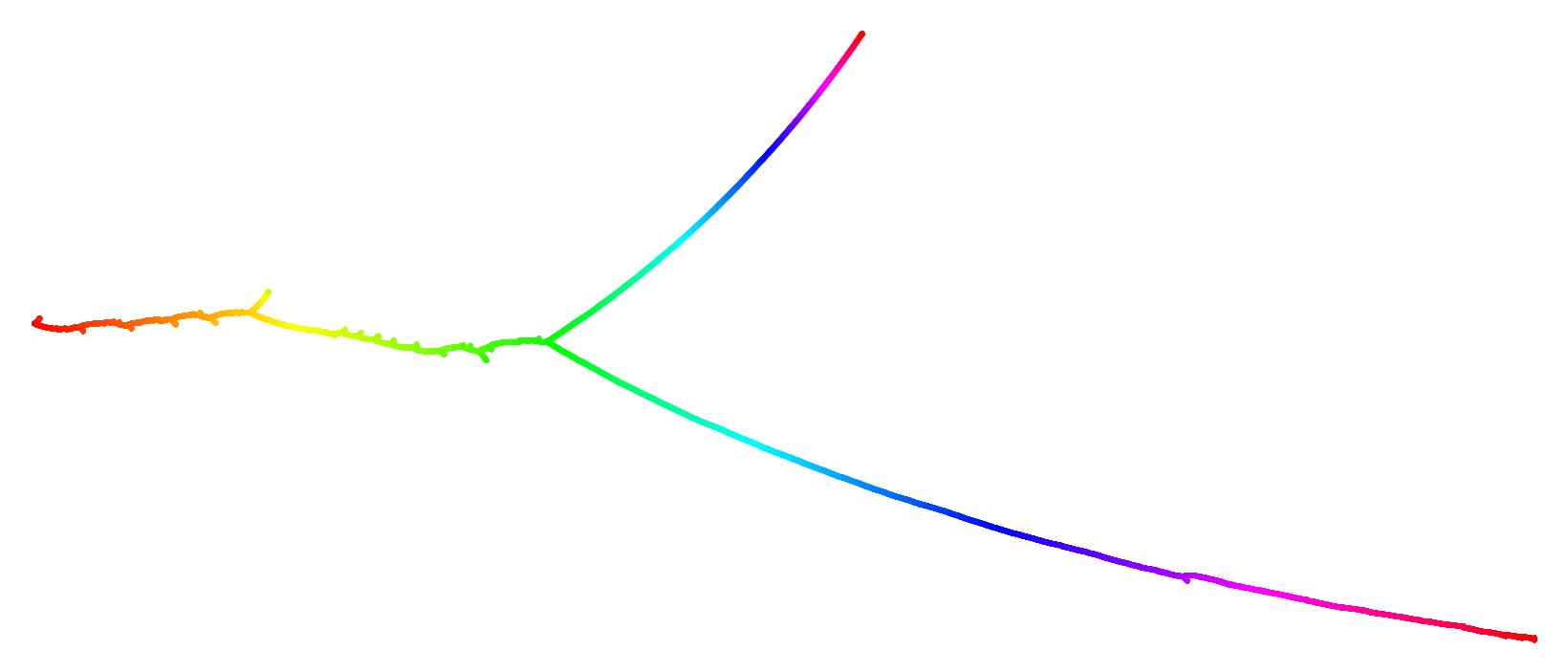} }\\
  \subfloat{\includegraphics[width=0.4\textwidth]{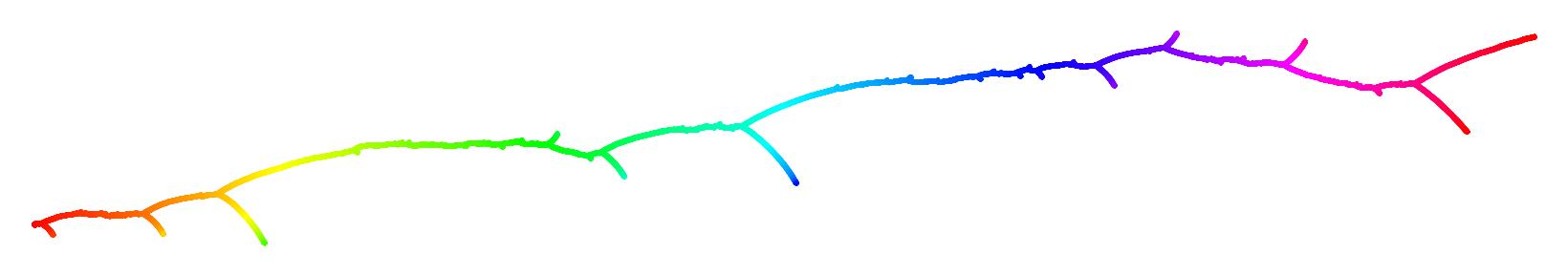} }\\
  \subfloat{\includegraphics[width=0.4\textwidth]{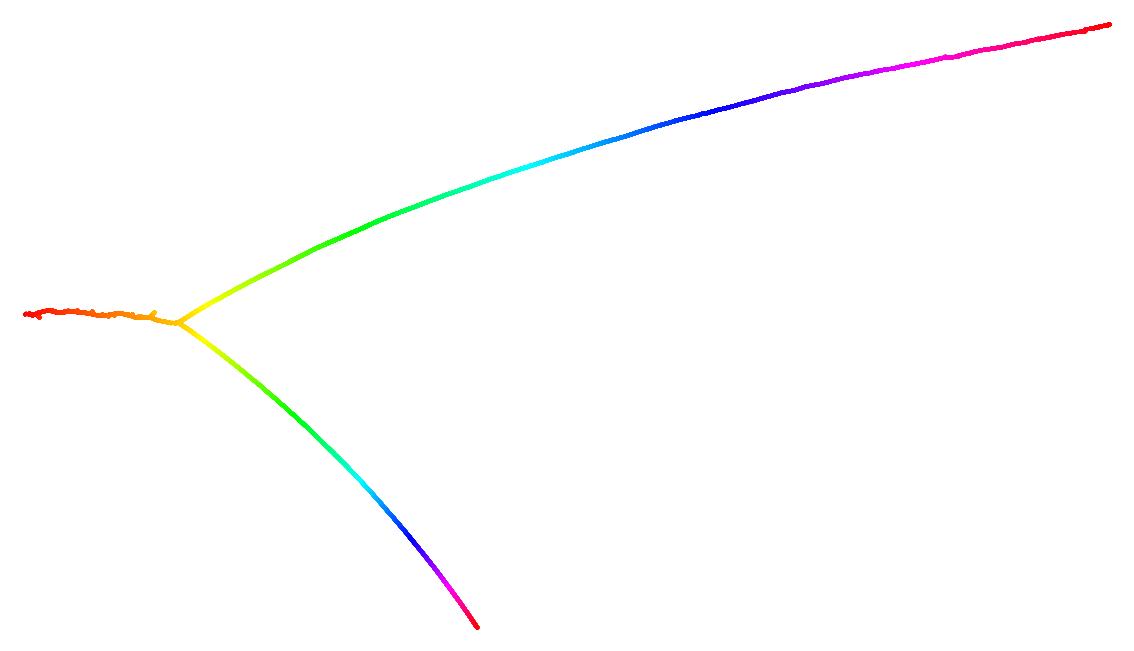} }\\
  \subfloat{\includegraphics[width=0.4\textwidth]{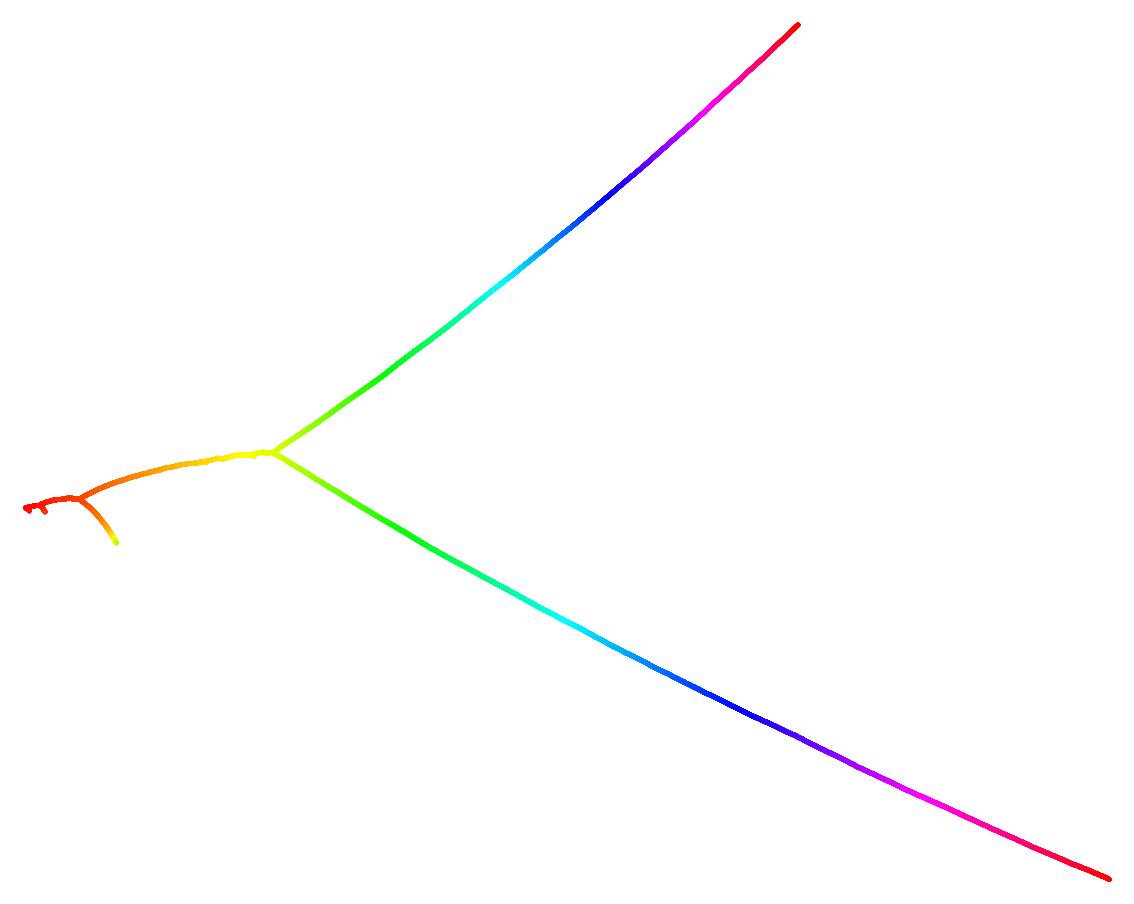} }
  \caption{Simulations of branching intensity. Each cluster contains $10^5$ discs. Colors indicate the order of attachment. The clusters represent different angles $\theta$, starting from the top:  $0.335\pi$, $0.34\pi$, $0.345\pi$, $0.35\pi$, $0.355\pi$, $0.36\pi$, $0.365\pi$, $0.37\pi$.} \label{fig_branch_high}
\end{figure}

\begin{figure}
  \centering
  \subfloat{\includegraphics[width=0.5\textwidth]{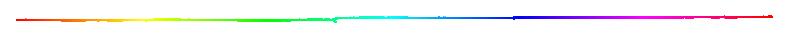} } \\
  \subfloat{\includegraphics[width=0.5\textwidth]{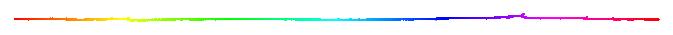} }\\
  \subfloat{\includegraphics[width=0.5\textwidth]{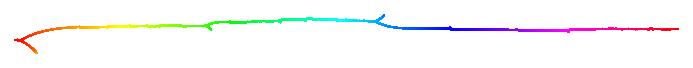} }\\
  \subfloat{\includegraphics[width=0.5\textwidth]{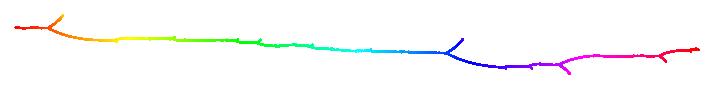} }\\
  \subfloat{\includegraphics[width=0.5\textwidth]{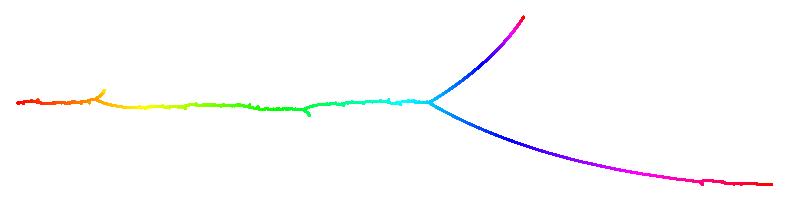} }\\
  \subfloat{\includegraphics[width=0.5\textwidth]{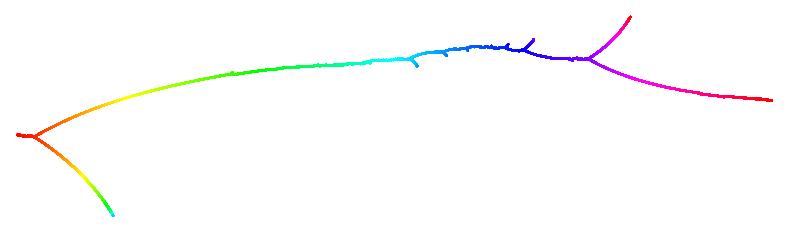} }\\
  \subfloat{\includegraphics[width=0.5\textwidth]{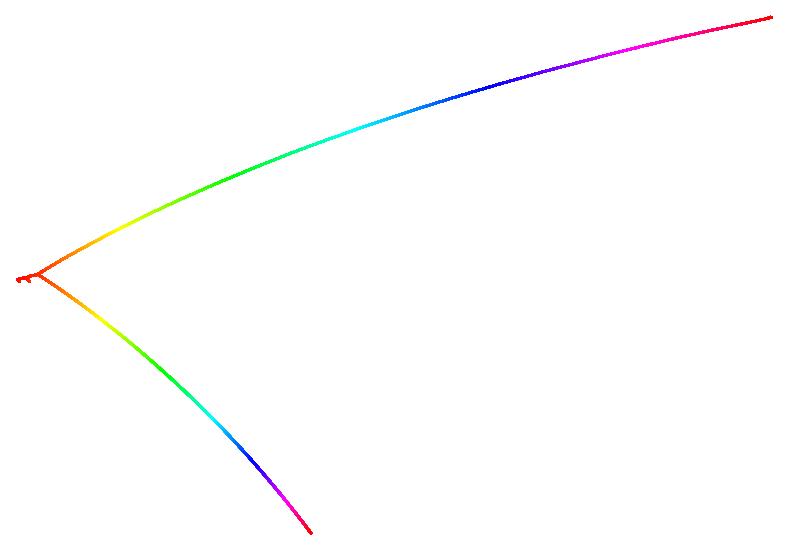} }\\
  \subfloat{\includegraphics[width=0.5\textwidth]{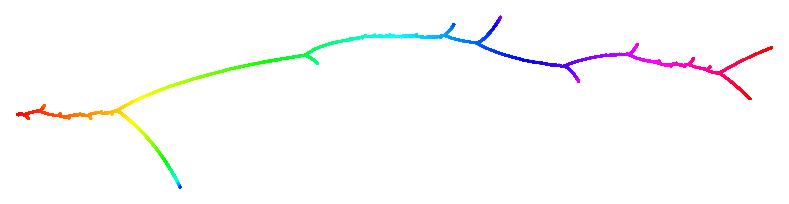} }
  \caption{Simulations of branching intensity. Each cluster contains $10^6$ discs. Colors indicate the order of attachment. The clusters represent different angles $\theta$, starting from the top: $0.3\pi$,   $0.31\pi$,  $0.32\pi$,  $0.33\pi$,  $0.34\pi$,  $0.35\pi$,  $0.36\pi$,  $0.37\pi$.} \label{fig_branch10_6}
\end{figure}

\section{Qualitative properties of $X_t$}\label{n1.6}

We will make some informal remarks on the properties of the process $X_t$ analyzed in Section \ref{secOU}.
We will refer to the results and use the notation from the proof of Theorem \ref{o27.7}.

\subsection{Conditional distribution for surviving paths at a large time}\label{n1.1}

The transformations of space and time in Section \ref{secOU} are deterministic
so, for large $t$, the shape of the distribution of $X_t$ conditioned on not hitting $\pm a$
until $t$
is the same as for Brownian motion $B$ at the corresponding time,
conditioned not to hit the transformed barrier until the transformed
time $c_*(1-t^{-2(\mu+1/2) })$,
see \eqref{o22.6}. This conditional distribution is close to the uniform distribution on $[-a,a]$. This claim captures informally the 
essence of the second part of the proof of Theorem \ref{o27.7} (i.e., the proof of the lower bound in \eqref{o22.7}.

\subsection{Typical oscillations of trajectories well before the exit time}\label{n1.2}

Fix a small $v>0$ and condition $B$ as in \eqref{o31.3}, assuming that 
$1-t^{-2(\mu+1/2) }\ll v$. The second part of the proof of Theorem \ref{o27.7}
shows that the process $\{B\left( c_*-v\right),v\in[c_*t^{-2(\mu+1/2) },c_*)\}$
is essentially a Brownian bridge. Hence, the distribution of $B\left( c_*-v\right)$ is close 
to centered normal with variance about $v$. Transforming this back to the 
language of $X_t$ shows that, for a fixed $t$, the standard deviation of $X_t$ is about $ t^{-1/2}$ if we condition $X$
not to hit $\pm a$ until $s\gg t$.

\subsection{Comparison with the standard Ornstein-Uhlenbeck process}

Consider a diffusion $R_t$ satisfying the SDE
\begin{align*}
    dR_t = a\, dB_t + b\, R_t dt,
\end{align*}
where $B$ is Brownian motion, and $a>0$ and $b\in \R$ are constants.
If $b<0$ then this process is called Ornstein-Uhlenbeck. If $b>0$, we call it
anti-Ornstein-Uhlenbeck. No matter whether $b$ is positive or negative,
the behavior of $R_t$ is totally different from that of $X_t$ discussed
in Sections \eqref{n1.1} and \eqref{n1.2}. 
First, if $t$ is very large and $R$ is conditioned not to hit $\pm a$
until $t$, the distribution of $R_t$ is not close to uniform on $[-a,a]$.
Second, if $s\gg t \gg 0$ and $R$ is conditioned not to hit $\pm a$ until
$s$ then the distribution of $R_t$ has standard deviation on the order of $a$.

\section{Polygonal convex hulls}\label{n1.7}

We will argue informally that there can be no stable polygonal shapes
with obtuse and distinct angles. More precisely, if $C_n' $ is the convex hull $C_n$
rescaled to have diameter 1, the probability that $C_n'$ will converge
to a polygon with obtuse and distinct angles is 0. 

Suppose that $n$ is very large and $C_n$ has a finite number of extremal discs, each touching two edges in $\prt C_n$ that form an obtuse angle. Consider two consecutive discs with endpoints at $x_1$ and $x_2$. Since we assume that all angles between edges in $\prt C_n$ are distinct, we can orient the coordinate system as in Fig. \ref{fig4}, so that the same angle $\gamma>0$ appears on both sides. By convention, $\beta>0$ in Fig. \ref{fig4}. This means that the angle between the edges of $\prt C_n$ at $x_1$ is more obtuse than that at $x_2$.

\begin{figure} [h]
\includegraphics[width=0.7\linewidth]{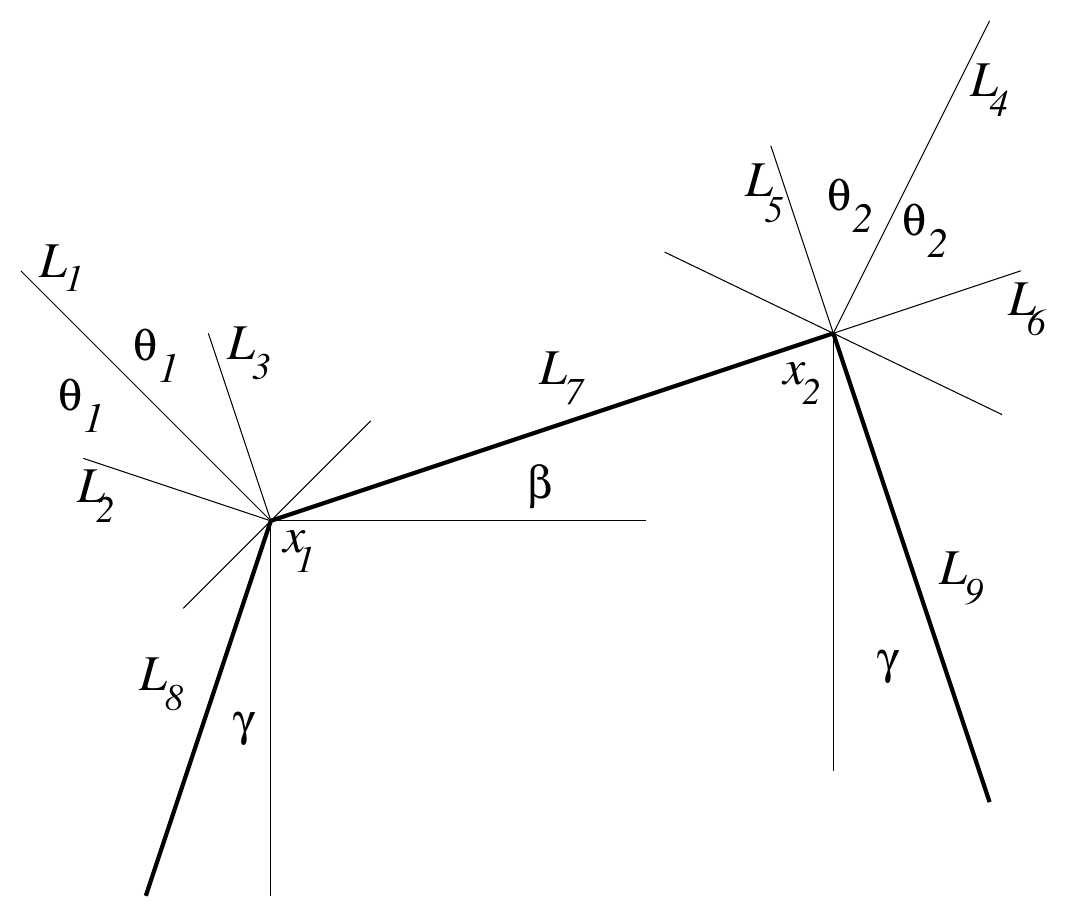}
\caption{ Thick lines represent the convex hull. The following pairs of lines are orthogonal: $L_8$ and $L_2$,  $L_3$ and $L_7$,  $L_7$ and $L_5$,  $L_6$ and $L_9$.  }
\label{fig4}
\end{figure}

Suppose that the distance between $x_1$ and $x_2$  is $r$. We will find a differential equation for $\beta$. The extremal points in the convex hull $x_1$ and $x_2$ will move along lines $L_1$ and $L_2$.  A  calculation similar to those
in Theorem \ref{o30.10} shows that if one disc is attached to the cluster, the average position of a new disc attached to the disc with center $x_1$ will be centered on $L_1$ and the average distance of its center from $x_1$ will be
\begin{align}\label{f18.3}
    &\frac{1}{2\pi} 2\int_0^{\theta_1} \cos \alpha d \alpha
    =
    \frac{1}{\pi} \sin{\theta_1},
\end{align}
and a similar claim holds for a disc attached to the disc centered at $x_2$, with $\theta_1$ replaced with $\theta_2$.
Informally,  the extremal discs with centers $x_1$ and $x_2$ will move along lines $L_1$ and $L_4$ with speeds $(1/\pi) \sin\theta_1$ and
$(1/\pi) \sin\theta_2$.
We did not condition on the new disc being attached to the disc centered at $x_1$ in \eqref{f18.3}.

The projection of the velocity vector of $x_1$ on $L_3$ has length $\frac 1 \pi \sin \theta_1 \cos \theta_1 = \frac 1 {2\pi} \sin(2\theta_1)$. We have $2\theta_1 + \beta+\gamma + 3\cdot \pi/2 = 2\pi$. Hence $2\theta_1 = \pi/2-\beta-\gamma$ and the projection of the velocity vector of $x_1$ on $L_3$ has length 
\begin{align}\label{f18.1}
    \frac 1 {2\pi} \sin(2\theta_1)
    =\frac 1 {2\pi} \sin(\pi/2-\beta-\gamma)
    =\frac 1 {2\pi} \cos(\beta+\gamma).
\end{align}
The analogous calculation for $x_2$ gives the velocity along $L_5$ equal to $ \frac 1 {2\pi} \sin(2\theta_2)$. We have $2\theta_2 +\gamma + 3\cdot \pi/2 = 2\pi+\beta$. Hence $2\theta_2 = \pi/2+\beta-\gamma$ and the projection of the velocity vector of $x_2$ on $L_5$ has a length 
\begin{align}\label{f18.2}
    \frac 1 {2\pi} \sin(2\theta_2)
    =\frac 1 {2\pi} \sin(\pi/2+\beta-\gamma)
    =\frac 1 {2\pi} \cos(-\beta+\gamma).
\end{align}
The rate of change of $\beta$ is the difference between the quantities in \eqref{f18.2} and \eqref{f18.1} divided by $r$. Thus
\begin{align}\label{f18.5}
    \frac{d\beta}{dn} = \frac{1}{r} \left(\frac 1 {2\pi} \cos(-\beta+\gamma) - \frac 1 {2\pi} \cos(\beta+\gamma)\right)
    = \frac{1}{\pi r} \sin \beta \sin \gamma.
\end{align}
Since $\beta>0$, the edge between $x_1$ and $x_2$ turns counterclockwise around $x_1$.

If $x_1$ is the center of a disc where $\prt C_n$ has the most obtuse angle then the angle
called $\gamma$
on the left in Fig. \ref{fig4} will also increase. Hence, the most obtuse angle will become even more obtuse, until it becomes $\pi$.
This completes the argument in support of our claim. Our proof is not completely rigorous because we applied the Law of Large Numbers without justification. Yet we believe that the argument can be made fully rigorous.

Note that the above argument does not apply to polygons with identical obtuse angles (regular polygonal domains) because in this case $\beta=0$. Answering Problem \ref{o30.2a} (vi) will require a more subtle analysis, most likely based on the Central Limit Theorem.

\section{Data availability statement}

There are no data associated with this paper.

\section{Acknowledgments}

This paper is the result of extensive discussions with Stefan Steinerberger
whose models, presented in \cite{stefan}, inspired the current project.
I am grateful for his advice and encouragement.
I am also grateful to Sayan Banerjee, Tomoyuki Ichiba,
Aleksandar Mijatovi\'c and Ruth Williams for very helpful advice.


\end{document}